\documentclass[11pt]{article}
\def\thetitle{Fast construction on a restricted budget}

\usepackage{graphicx}

\usepackage{amsmath,amssymb}

\usepackage[usenames,dvipsnames,svgnames,table]{xcolor}
\definecolor{CombinatoricaAqua}{HTML}{00698C}
\definecolor{CombinatoricaBlue}{HTML}{3A3293}
\definecolor{CombinatoricaBrown}{HTML}{66220C}
\definecolor{CombinatoricaRed}{HTML}{DF2A27}
\definecolor{HarvardCrimson}{rgb}{0.6471, 0.1098, 0.1882}
\definecolor{DAGreen}{HTML}{339900}

\makeatletter
\let\reftagform@=\tagform@
\def\tagform@#1{\maketag@@@
	{(\ignorespaces\textcolor{CombinatoricaBrown}{#1}\unskip\@@italiccorr)}}
\renewcommand{\eqref}[1]{\textup{\reftagform@{\ref{#1}}}}
\makeatother

\usepackage[backref=page]{hyperref}
\hypersetup{%
	unicode,
	pdfencoding=auto,
	pdfauthor={Peleg Michaeli},
	pdftitle={\thetitle},
	pdfsubject={},
	pdfkeywords={},
	colorlinks=true,
	citecolor=CombinatoricaBlue,
	linkcolor=CombinatoricaAqua,
	anchorcolor=CombinatoricaBrown,
	urlcolor=HarvardCrimson}

\usepackage{amsthm}
\usepackage{thmtools}
\usepackage{bbm}
\usepackage{enumitem}
\usepackage[capitalize]{cleveref}
\Crefname{mainthm}{Theorem}{Theorems}
\Crefname{fact}{Fact}{Facts}
\Crefname{claim}{Claim}{Claims}


\declaretheoremstyle[
spaceabove=\topsep, spacebelow=\topsep,
headfont=\color{CombinatoricaBrown}\normalfont\bfseries,
bodyfont=\itshape,
]{thm}
\declaretheoremstyle[
spaceabove=\topsep, spacebelow=\topsep,
headfont=\color{CombinatoricaBrown}\normalfont\bfseries,
bodyfont=\normalfont,
]{dfn}
\declaretheoremstyle[
spaceabove=0.5\topsep, spacebelow=0.5\topsep,
headfont=\color{CombinatoricaBrown}\normalfont\bfseries,
bodyfont=\normalfont,
]{rmk}

\declaretheorem[style=thm,name=Theorem,parent=section]{mainthm}
\declaretheorem[style=thm,sibling=mainthm]{theorem}
\declaretheorem[style=thm,sibling=theorem]{lemma}
\declaretheorem[style=thm,sibling=theorem]{corollary}
\declaretheorem[style=thm,sibling=theorem]{claim}
\declaretheorem[style=thm,sibling=theorem]{proposition}

\declaretheorem[style=thm,sibling=theorem]{observation}
\declaretheorem[style=thm,sibling=mainthm]{conjecture}

\declaretheorem[style=rmk,numbered=no]{remark}

\usepackage[nobysame,msc-links,non-sorted-cites]{amsrefs}

\renewcommand{\eprint}[1]{\href{https://arxiv.org/abs/#1}{arXiv:#1}}

\BibSpec{book}{%
	+{}  {\PrintPrimary}                {transition}
	+{,} { \textbf}                     {title} 
	+{.} { }                            {part}
	+{:} { \textit}                     {subtitle}
	+{,} { \PrintEdition}               {edition}
	+{}  { \PrintEditorsB}              {editor}
	+{,} { \PrintTranslatorsC}          {translator}
	+{,} { \PrintContributions}         {contribution}
	+{,} { }                            {series}
	+{,} { \voltext}                    {volume}
	+{,} { }                            {publisher}
	+{,} { }                            {organization}
	+{,} { }                            {address}
	+{,} { \PrintDateB}                 {date}
	+{,} { }                            {status}
	+{}  { \parenthesize}               {language}
	+{}  { \PrintTranslation}           {translation}
	+{;} { \PrintReprint}               {reprint}
	+{.} { }                            {note}
	+{.} {}                             {transition}
	+{}  {\SentenceSpace \PrintReviews} {review}
}
\BibSpec{incollection}{%
  +{}  {\PrintAuthors}                {author}
  +{,} { \textit}                     {title}
  +{.} { }                            {part}
  +{:} { \textit}                     {subtitle}
  +{,} { \PrintContributions}         {contribution}
  +{,} { \PrintConference}            {conference}
  +{}  {\PrintBook}                   {book}
  +{,} { }                            {booktitle}
	+{}  { \PrintEditorsB}              {editor}
	+{,} { }                            {publisher}
  +{,} { \PrintDateB}                 {date}
  +{,} { pp.~}                        {pages}
  +{,} { }                            {status}
  +{,} { \PrintDOI}                   {doi}
  +{,} { available at \eprint}        {eprint}
  +{}  { \parenthesize}               {language}
  +{}  { \PrintTranslation}           {translation}
  +{;} { \PrintReprint}               {reprint}
  +{.} { }                            {note}
  +{.} {}                             {transition}
  +{}  {\SentenceSpace \PrintReviews} {review}
}

\makeatletter
\renewcommand{\PrintNames@a}[4]{%
	\PrintSeries{\name}
	{#1}
	{}{ and \set@othername}
	{,}{ \set@othername}
	{}{ and \set@othername}
	{#2}{#4}{#3}%
}
\makeatother

\makeatletter
\def\mathcolor#1#{\@mathcolor{#1}}
\def\@mathcolor#1#2#3{%
	\protect\leavevmode
	\begingroup
	\color#1{#2}#3%
	\endgroup
}
\makeatother
\definecolor{Red}{rgb}{0.618,0,0}
\definecolor{Blue}{rgb}{0,0,1}
\definecolor{Green}{rgb}{0,0.298,0}

\usepackage{sectsty}
\sectionfont{\color{CombinatoricaBrown}}
\subsectionfont{\color{CombinatoricaBrown}}
\subsubsectionfont{\color{CombinatoricaBrown}}

\usepackage{soul}
\soulregister\ref7

\usepackage{pifont}
\usepackage{calc}

\usepackage[a4paper]{geometry}
\title{\thetitle}
\author{
  Alan Frieze\thanks{
    Department of Mathematical Sciences,
    Mellon College of Science,
    Carnegie Mellon University,
    Pittsburgh, PA, USA.
    Email: \href{mailto:frieze@cmu.edu}
                {\tt frieze@cmu.edu}.
    Research supported in part by NSF grant DMS1952285.
  }
  \and
  Michael Krivelevich\thanks{
    School of Mathematical Sciences,
    Tel Aviv University,
    Tel Aviv 6997801, Israel.
    Email: \href{mailto:krivelev@tauex.tau.ac.il}
                {\tt krivelev@tauex.tau.ac.il}.
    Research supported in part by USA--Israel BSF grant 2018267.
  }
  \and
  Peleg Michaeli\thanks{
    Department of Mathematical Sciences,
    Mellon College of Science,
    Carnegie Mellon University,
    Pittsburgh, PA, USA.
    Email: \href{mailto:pelegm@cmu.edu}
                {\tt pelegm@cmu.edu}.
    Research partially supported by NSF grant DMS1952285.
  }
}

\makeatletter
\def\namedlabel#1#2{\begingroup
  #2%
  \def\@currentlabel{#2}%
  \phantomsection\label{#1}\endgroup
}
\makeatother

\usepackage{mleftright}
\mleftright

\newcommand{\defn}[1]{{\bfseries #1}}

\newcommand{\eps}{\varepsilon}
\renewcommand{\phi}{\varphi}

\newcommand{\cH}{\mathcal{H}}

\newcommand{\cP}{\mathcal{P}}
\newcommand{\cR}{\mathcal{R}}

\newcommand{\co}{\mathfrak{o}}

\newcommand{\floor}[1]{\left\lfloor{#1}\right\rfloor}
\newcommand{\ceil}[1]{\left\lceil{#1}\right\rceil}

\DeclareMathOperator{\ent}{H}

\newcommand{\sm}{\smallsetminus}
\newcommand{\es}{\varnothing}

\newcommand{\pr}[0]{\mathbb{P}}
\newcommand{\E}[0]{\mathbb{E}}

\newcommand{\whp}[0]{\textbf{whp}}
\newcommand{\Whp}[0]{\textbf{Whp}}

\newcommand{\Dist}[1]{\mathsf{#1}}
\newcommand{\Bin}{\Dist{Bin}}

\usepackage{tabto}

\usepackage{marginnote}

\usepackage{todonotes}
\reversemarginpar

\newcommand{\stage}[2]{\hfill{\small{\textcolor{CombinatoricaAqua}{Time:} #1 | \textcolor{DAGreen}{Budget:} #2}}%
                       \vspace{2pt plus 2pt minus 2pt}\\}

\newcommand{\cG}{\mathcal{G}}
\newcommand{\cO}{\mathcal{O}}
\newcommand{\w}{\mathbf{w}}
\newcommand{\x}{\mathbf{x}}

\usepackage{comment}

\usepackage{subcaption}
\usepackage{tikz,pgfplots}
\usetikzlibrary{calc}

\begin{document}
\maketitle
\thispagestyle{empty}

\begin{abstract}
  \normalsize 
  We introduce a model of a controlled random graph process.
  In this model,
  the edges of the complete graph $K_n$ are ordered randomly and then revealed,
  one by one,
  to a player called Builder.
  He must decide,
  immediately and irrevocably,
  whether to purchase each observed edge.
  The observation time is bounded by parameter $t$,
  and the total budget of purchased edges is bounded by parameter $b$.
  Builder's goal is to devise a strategy that,
  with high probability,
  allows him to construct a graph of purchased edges possessing a target graph property $\cP$,
  all within the limitations of observation time and total budget.
  We show the following:
  \begin{itemize}
    \item Builder has a strategy to achieve $k$-vertex-connectivity at the hitting time for this property
      by purchasing at most $c_kn$ edges for an explicit $c_k<k$;
      and a strategy to achieve minimum degree $k$ (slightly) after the threshold for minimum degree $k$
      by purchasing at most $(1+\eps)kn/2$ edges (which is optimal).
    \item Builder has a strategy to create a Hamilton cycle
      at the hitting time for Hamiltonicity by purchasing at most $Cn$ edges
      for an absolute constant $C>1$;
      this is optimal in the sense that $C$ cannot be arbitrarily close to $1$.
      This substantially extends the classical hitting time result for Hamiltonicity
      due to Ajtai--Koml\'os--Szemer\'edi and Bollob\'as.
    \item Builder has a strategy to create a perfect matching
      by time $(1+\eps)n\log{n}/2$ while purchasing at most $(1+\eps)n/2$ edges
      (which is optimal).
    \item Builder has a strategy to create a copy of a given $k$-vertex tree
      if $t\ge b\gg\max\{(n/t)^{k-2},1\}$, and this is optimal;
    \item For $\ell=2k+1$ or $\ell=2k+2$,
      Builder has a strategy to create a copy of a cycle of length $\ell$
      if $b\gg\max \{n^{k+2}/t^{k+1},n/\sqrt{t}\}$, and this is optimal.
  \end{itemize}
\end{abstract}

\newpage
\section{Introduction}
\subsection{The model}
The random graph process, introduced by Erd\H{o}s and R\'enyi~\cites{ER59,ER60},
is a stochastic process that starts with an empty $n$-vertex graph
and, at each step, gains a new uniformly selected random edge.
At any fixed time $t$, the process is distributed as the uniform random graph $G(n,t)$.
A \defn{graph property} is a family of graphs
that is closed under isomorphisms.
It is \defn{monotone} if it is closed under addition of edges.
A vast body of literature is concerned with finding {\em thresholds}
for various monotone graph properties in the random graph process,
namely, with finding time $t_c$ such that the random graph process belongs to $\cP$
with high probability
(\whp{}; that is, with probability tending to $1$ as $n\to\infty$)
whenever $t$ is much (or somewhat) larger than $t_c$
and does not belong to $\cP$ \whp{} if $t$ is much (or somewhat) smaller than $t_c$.

In many cases, if one observes the random graph process at time $t$ above the threshold,
the graph has the desired property but, in fact, contains a much sparser subgraph that has the property.
For example, 
one of the outstanding results in this model regards a threshold for Hamiltonicity:
Koml\'os and Szemer\'edi~\cite{KS83} and, independently, Bollob\'as~\cite{Bol84}
showed that if $2t/n-\log{n}-\log\log{n}$ tends to infinity,
then the random graph process, at time $t$, contains a Hamilton cycle \whp{}.
Evidently, not all $\sim n\log{n}/2$ observed edges must be included in the resulting graph
for it to be Hamiltonian
(as any Hamilton cycle only uses $n$ edges).
Nevertheless, when an edge arrives, it is generally hard to determine whether it
will be crucial for Hamiltonicity.

The above motivates the following ``online'' version of building a subgraph of the random graph process.
We think of it as a one-player game, where the player (``\defn{Builder}'')
has a limited ``budget''.
Edges ``arrive'' one at a time in random order and are presented to Builder.
Whenever he {\em observes} an edge,
he must immediately and irrevocably decide whether to {\em purchase} it.
A non-purchased edge is thrown away and never reappears.
The \defn{time} (total number of presented edges)
and the \defn{budget} (maximum number of edges Builder can purchase)
are both capped.
The question is as to whether or not, under the given time and budget constraints,
Builder has a strategy that allows him to obtain,
\whp{},
a particular monotone graph property in the graph of purchased edges.

To demonstrate the model, consider the property of connectedness.
Let $\tau_C$ be the (random) time at which the random graph process becomes connected.
Evidently, if Builder wishes to construct a connected subgraph of the random graph process,
he must purchase at least $n-1$ edges.
However, in this case purchasing $n-1$ edges suffices:
Builder's strategy would be to purchase an edge if and only if it decreases the number of connected components in his graph.
That way, Builder maintains a forest, which becomes connected exactly at time $\tau_C$.
Therefore, in this example, Builder does not have to ``pay'' for having to make decisions online.
However, this is not always the case.
For example, if Builder wants to purchase a triangle
and wishes to do so while observing $o(n^2)$ edges,
he must purchase (many) more than three edges (see \cref{thm:cycles}).

~

We denote the underlying random graph process at time $t$
(namely, after $t$ edges have been presented to Builder)
by $G_t$
(where the number of vertices, $n$, is implicit).
As we mentioned earlier, $G_t$ is distributed as the uniform random graph $G(n,t)$.
The \defn{hitting time} for a monotone graph property $\cP$
is the (random) minimum time $t$ for which $G_t$ has $\cP$.
Builder's graph at time $t$, denoted $B_t$, is a subgraph of $G_t$ on the same vertex set
that consists of the edges purchased by Builder by time $t$.
A \defn{$(t,b)$-strategy} of Builder is a (potentially random) function that,
for any $s\le t$,
decides whether to purchase the edge presented at time $s$,
given $B_{s-1}$,
under the limitation that $B_s$ has at most $b$ edges.

\subsection{Our results}
Our first result discusses a strategy for constructing a subgraph with a given minimum degree
and, for $k\ge 3$, of a given vertex-connectivity,
as quickly as possible.
For every positive integer $k$,
denote by $\tau_k$ the hitting time for minimum degree $k$ in the random graph process.
It was proved by Erd\H{o}s and R\'enyi~\cite{ER61} that
$\tau_k=\frac{n}{2}(\log{n}+(k-1)\log\log{n}+O(1))$ \whp{}
(see \cref{lem:gnt:deg:hit}).
Let $\kappa$ denote vertex-connectivity.

\begin{mainthm}[Minimum degree and vertex-connectivity at the hitting time]\label{thm:mindeg:hit}
  For every positive integer $k$
  there exists a constant $\co_k\in(k/2,3k/4]$
  such that the following holds.
  If $b\ge \co_k n+\omega(\sqrt{n}\log{n})$
  then there exists a $(\tau_k,b)$-strategy $B$ of Builder such that
  \[
    \lim_{n\to\infty} \pr(\delta(B_{\tau_k})\ge k) = 1.
  \]
  If $k\ge 3$, the same strategy guarantees
  \[
    \lim_{n\to\infty} \pr(\kappa(B_{\tau_k})\ge k) = 1.
  \]
\end{mainthm}
For $k=1$, as we mentioned in the introduction,
a sufficient and necessary budget for connectivity at the hitting time is $n-1$.
For $k=2$, \cref{thm:ham:hit} below implies that a budget of $O(n)$ suffices
for $2$-connectivity at the hitting time for minimum degree $2$.
The constant $\co_k$ in \cref{thm:mindeg:hit} is explicit and computable
(see \cref{eq:ok} in \cref{sec:ok};
 the first three values in the sequence $\co_k$ are $\frac{3}{4}$, $\frac{11}{8}$ and $\frac{63}{32}$),
and is roughly $k/2$ when $k$ is large (see \cref{cor:ok:size}).
We believe that at the hitting time $\tau_k$, that constant is optimal
(see \cref{conj:mindeg} and the discussion following it).
However,
if we allow the time to be optimal only asymptotically,
then an asymptotically optimal budget suffices.
\begin{mainthm}[Minimum degree]\label{thm:mindeg}
  Let $k$ be a positive integer and let $\eps>0$.
  If $t\ge(1+\eps)n\log{n}/2$ and $b\ge (1+\eps)kn/2$
  then there exists a $(t,b)$-strategy $B$ of Builder such that
  \[
    \lim_{n\to\infty} \pr(\delta(B_t)\ge k) = 1.
  \]
\end{mainthm}

The above theorem is tight in the following sense:
if $t\le(1-\eps)n\log{n}/2$ then the underlying random graph process has, \whp{},
isolated vertices;
and if $b<kn/2$ then every strategy will fail
since Builder will not be able to purchase enough edges to obtain the required minimum degree.
We remark that
after an earlier version of this paper appeared online,
a vertex-connectivity version of \cref{thm:mindeg} that holds for every $k\ge 2$
has been proved by Lichev~\cite{Lic22+} (see \cref{sec:conclude}).
\Cref{thm:mindeg} follows 
from \cref{thm:pm} below (for $k=1$)
and 
from Lichev's result (for $k\ge 2$).
We keep it here due to its short and simple proof (see \cref{sec:mindeg}).
Note that \cref{thm:mindeg:hit} is {\em not} implied by Lichev's result,
as \cref{thm:mindeg:hit} implies a strategy that succeeds \whp{} {\em at the hitting time} $\tau_k$,
whereas Lichev's strategy typically requires $\omega(n)$ additional steps
(which can, as we demonstrate later in \cref{thm:ham:hit} and the discussion following it,
make a big difference).
An asymptotic version of Lichev's result,
according to which for every $\eps>0$ there exists $k_0$
such that for every $k\ge k_0$
there exists a strategy to build a $k$-vertex-connected graph
in time $(1+\eps)n\log{n}/2$
and with budget $(1+\eps)kn/2$
follows from \cref{thm:mindeg:hit}
together with the observation that $\co_k\sim k/2$ when $k$ is large
(\cref{cor:ok:size}).

\medskip

We continue to Hamiltonicity.
The classical {\em hitting-time} result of Bollobás~\cite{Bol84}
and independently of Ajtai, Koml\'os, and Szemer\'edi~\cite{AKS85}
states that the random graph process, \whp{},
becomes Hamiltonian as soon as its minimum degree reaches $2$.
As we reminded earlier, $\tau_2=\Theta(n\log{n})$,
while a Hamilton cycle has only $n$ edges.
The following result states that there exists an online algorithm
choosing only linearly many edges from the first $\tau_2$ edges of the random graph process
that still obtains a Hamilton cycle \whp{}.
Moreover, as we show in the proof,
this algorithm is {\em extremely} simple, deterministic, and requires only local information.
The theorem recovers -- and substantially strengthens -- the aforementioned classical result.
\begin{mainthm}[Hamiltonicity at the hitting time]\label{thm:ham:hit}
  There exists $C>1$ for which there exists a $(\tau_2,Cn)$-strategy $B$ of Builder with
  \[
    \lim_{n\to\infty} \pr(B_{\tau_2}\text{ is Hamiltonian})=1.
  \]
\end{mainthm}
The constant $C$ we obtain in \cref{thm:ham:hit} is large,
and we made no serious effort to optimise it.
A natural question is whether it could be as low as $1+\eps$ for every $\eps>0$.
The answer, found by Anastos~\cite{Ana22+} (see the concluding remarks there), is negative.
As part of the proof of \cref{thm:ham:hit},
we show that the so-called {\em random $k$-nearest neighbour graph}
(see \cref{sec:ok})
-- for large enough constant $k$ --
is \whp{} Hamiltonian
(\cref{thm:ok:ham}).
This partially solves an open problem of the first author~\cite{FriBib}*{Problem 45}.
To prove \cref{thm:ham:hit},
we prove a stronger hitting time version of \cref{thm:ok:ham} (\cref{thm:ok:ham:hit})
which states that the $k$-nearest neighbour graph,
considered as a process and {\em stopped as the minimum degree becomes $2$},
is \whp{} Hamiltonian.

We remark that an immediate corollary of \cref{thm:ham:hit}
(using \cref{lem:gnt:deg:hit})
is that if $t\ge(1+\eps)n\log{n}/2$ (namely, slightly after the Hamiltonicity threshold)
and $b\ge Cn$ (namely, when the budget is inflated by a constant),
then there exists a $(t,b)$-strategy of Builder that succeeds (\whp{}) in building a Hamilton cycle.
In an earlier version of this paper that appeared online,
we also proved a complementary statement:
if $t\ge Cn\log{n}/2$ and $b\ge (1+\eps)n$,
then there exists a ``successful'' $(t,b)$-strategy
(see \cref{thm:ham}).
After that early version appeared online,
Anastos~\cite{Ana22+} proved, using different and more involved techniques,
that Hamiltonicity can be achieved in asymptotically optimal time {\em and}
budget.
Namely, he showed that in the above statement, one could take $C=C(\eps)=1+\eps$.
He states this result as a special case
in a model that generalises both the budget model discussed here
and the so-called Achlioptas process
(see discussion in \cref{sec:related}).
We leave here our original proof due to its relative shortness and simplicity (see \cref{sec:ham:inflated}).
Note that \cref{thm:ham:hit} does not follow from the work of Anastos.

\medskip

For perfect matchings,
we show that there exists an asymptotically-optimal-time asymptotically-optimal-budget strategy.

\begin{mainthm}[Perfect matchings]\label{thm:pm}
  Suppose $n$ is even.
  For every $\eps>0$
  if $t\ge(1+\eps)n\log{n}/2$ and $b\ge(1+\eps)n/2$
  then there exists a $(t,b)$-strategy $B$ of Builder such that
  \[
    \lim_{n\to\infty}\pr(B_t\text{ has a perfect matching})=1.
  \]
\end{mainthm}

  An earlier version of this paper that appeared online
  contained a non-optimal version of \cref{thm:pm}
  (analogous to \cref{thm:ham}).
  Recently,
  Anastos~\cite{Ana22+} proved a more general version of \cref{thm:pm}
  (see remark after \cref{thm:ham}),
  using different and more involved techniques.
  The proof we provide here is more elementary.

\medskip

The next two theorems discuss optimal strategies for purchasing small subgraphs.
We resolve the problem whenever the target subgraph is a fixed tree or a cycle.
\begin{mainthm}[Small trees]\label{thm:trees}
  Let $k\ge 3$ be an integer and let $T$ be a $k$-vertex tree.
  If 
  $t\ge b\gg \max\{(n/t)^{k-2},1\}$
  then there exists a $(t,b)$-strategy $B$ of Builder such that
  \[
    \lim_{n\to\infty} \pr(T\subseteq B_t) = 1
  \]
  and if $b\ll (n/t)^{k-2}$
  then for any $(t,b)$-strategy $B$ of Builder,
  \[
    \lim_{n\to\infty} \pr(T\subseteq B_t) = 0.
  \]
\end{mainthm}
In simple (and slightly imprecise) terms,
\cref{thm:trees} states that the {\em budget threshold}
for a $k$-vertex tree (for $k\ge 3$) is $b^*=(n/t)^{k-2}$.
On a logarithmic scale (with base $n$),
this corresponds to ${\log_n}{b^*}=(k-2)(1-{\log_n}{t})$.
This, along with the bound $b\le t$, is illustrated in \cref{fig:trees}.

\begin{figure*}[t!]
  \captionsetup{width=0.879\textwidth,font=small}
  \centering
  \begin{tikzpicture}
    \begin{axis}[
        xmin=0, xmax=1.5,
        ymin=0, ymax=1,
        x=5cm,y=5cm,
        samples=100,
        axis y line=center,
        axis x line=middle,
        xlabel={${\log_n}{t}$},
        xtick={0,1/2,2/3,3/4,4/5,1},
        xticklabels={0,$\frac{1}{2}$,$\frac{2}{3}$,,$\frac{4}{5}$,1},
        ylabel={${\log_n}{b}$},
        ytick={0,1/2,2/3,3/4,4/5},
        extra y ticks={1},
        extra tick style={tick style={draw=none}},
        yticklabels={0,1/2,2/3,,4/5,1},
        every axis x label/.style={
            at={(ticklabel* cs:1.05)},
            anchor=west,
        },
        every axis y label/.style={
            at={(ticklabel* cs:1.05)},
            anchor=south,
        },
        reverse legend,
    ]
        \addplot[color=gray,loosely dotted,thick,domain=0:1] {x};

        \addplot[color=purple,loosely dashdotted,thick,domain=4/5:1] {4-4*x};

        \addplot[color=ForestGreen,densely dashed,thick,domain=3/4:1] {3-3*x};

        \addplot[color=blue,densely dotted,thick,domain=2/3:1] {2-2*x};

        \addplot[color=red,thick,domain=1/2:1] {1-x};

        \legend{$x\mapsto x$,$k=6$,$k=5$,$k=4$,$k=3$}
    \end{axis}
  \end{tikzpicture}
  \caption{Budget thresholds for $k$-vertex trees.}
  \label{fig:trees}
\end{figure*}

\begin{remark}
  We show in the proof of \cref{thm:trees} that if $t\gg n$ then, in fact,
  there exists a $(t,b)$-strategy for $b=k-1$ that succeeds \whp{}
  in building a copy of a $k$-vertex tree.
\end{remark}

\begin{mainthm}[Short cycles]\label{thm:cycles}
  Let $k\ge 1$ be an integer and let $H=C_{2k+1}$ or $H=C_{2k+2}$.
  Write $b^*=b^*(n,t,k)=\max\{n^{k+2}/t^{k+1},n/\sqrt{t}\}$.
  If $t\gg n$ and $b\gg b^*$
  then there exists a $(t,b)$-strategy $B$ of Builder such that
  \[
    \lim_{n\to\infty} \pr(H\subseteq B_t) = 1,
  \]
  and if $t\ll n$ or $b\ll b^*$
  then for any $(t,b)$-strategy $B$ of Builder,
  \[
    \lim_{n\to\infty} \pr(H\subseteq B_t) = 0.
  \]
\end{mainthm}
Again, in simple (and slightly imprecise) terms,
\cref{thm:cycles} states that the {\em budget threshold}
for an $\ell$-cycle, where $\ell=2k+1$ or $\ell=2k+2$ (for $k\ge 1$),
is $b^*=\max\{n^{k+2}/t^{k+1},n/\sqrt{t}\}$.
On a logarithmic scale (with base $n$),
this corresponds to ${\log_n}{b^*}=\max\{k+2-(k+1){\log_n}{t},1-\frac{1}{2}{\log_n}{t}\}$.
This is illustrated in \cref{fig:cycles}.
For discussion on the difficulty arising in handling graphs 
with more than one cycle see \cref{sec:conclude}.

\begin{figure*}[t]
  \captionsetup{width=0.879\textwidth,font=small}
  \centering
  \begin{tikzpicture}
    \begin{axis}[
        xmin=0, xmax=2,
        ymin=0, ymax=1,
        x=5cm,y=5cm,
        samples=100,
        axis y line=center,
        axis x line=middle,
        xlabel={${\log_n}{t}$},
        xtick={0,1,10/9,8/7,6/5,4/3},
        extra x ticks={2},
        extra tick style={tick style={draw=none}},
        xticklabels={0,$1$,,,$\frac{6}{5}$,$\frac{4}{3}$,$2$},
        ylabel={${\log_n}{b}$},
        ytick={0,1/3,1/2-1/10,1/2-1/18,1/2},
        extra y ticks={1},
        yticklabels={0,1/3,2/5,,1/2,1},
        extra tick style={tick style={draw=none}},
        every axis x label/.style={
            at={(ticklabel* cs:1.05)},
            anchor=west,
        },
        every axis y label/.style={
            at={(ticklabel* cs:1.05)},
            anchor=south,
        },
        reverse legend,
    ]
        \addplot[color=gray,loosely dotted,thick,domain=0:2] {1-x/2};

        \addplot[color=purple,loosely dashdotted,thick,domain=1:10/9] {6-5*x};
        \addplot[color=purple,loosely dashdotted,thick,domain=10/9:2] {1-x/2};

        \addplot[color=ForestGreen,densely dashed,thick,domain=1:8/7] {5-4*x};
        \addplot[color=ForestGreen,densely dashed,thick,domain=8/7:2] {1-x/2};

        \addplot[color=blue,densely dotted,thick,domain=1:6/5] {4-3*x};
        \addplot[color=blue,densely dotted,thick,domain=6/5:2] {1-x/2};

        \addplot[color=red,thick,domain=1:4/3] {3-2*x};
        \addplot[color=red,thick,domain=4/3:2] {1-x/2};

        \legend{$x\mapsto 1-x/2$,,$k=4$,,$k=3$,,$k=2$,,$k=1$}
    \end{axis}
  \end{tikzpicture}
  \caption{Budget thresholds for $C_{2k+1},C_{2k+2}$.}
  \label{fig:cycles}
\end{figure*}

\subsection{Tools and techniques}
\paragraph{Expanders}
{\em Expanders} are graphs in which (sufficiently) small sets expand.
Namely,
these are graphs in which the neighbourhood of each small set is larger than that set by a constant factor%
\footnote{We need, and hence use, a rather weak notion of expansion; see \cref{sec:boost}.}.
It is well known that connected expanders are helpful in finding Hamilton cycles (see \cref{lem:boosters});
more concretely, connected non-hamiltonian expanders have many ``boosters'', namely,
non-edges whose addition to the graph creates a graph that is either Hamiltonian
or whose longest path is longer
(for a comprehensive account, we refer the reader to~\cite{Kri16}).
Thus, expanders will play a crucial role in the proof of \cref{thm:ham:hit}
(and also in the proof of \cref{thm:pm}):
Builder will attempt to construct (sparse, and therefore cheap)
connected expanders within the random graph process.
A standard method for obtaining sparse expanders in random graphs is choosing an appropriate random (sub)graph model.
Natural candidates are discussed in the following paragraphs.

\paragraph{Random $k$-out graphs}
In the goals described in \cref{thm:mindeg:hit,thm:mindeg,thm:ham:hit,thm:pm},
Builder must achieve a certain minimum degree in his (spanning) graph.
The (standard) random graph process is quite wasteful in this regard,
as to avoid isolated vertices,
a superlinear number of edges must arrive.
Thus, a wise Builder would instead construct a much sparser subgraph of the random graph process
with the desired minimum degree.
A classical sparse graph with a given minimum degree is the {\em random regular graph} (see, e.g., in~\cite{FK}).
However, such a graph is generally very hard to construct online.
A much simpler alternative is the so-called {\em random $k$-out graph} (see, e.g., in~\cite{FK}).
In the $k$-out graph, each vertex chooses, random, independently, and without repetitions,
$k$ neighbours to connect to.
Thus, the total number of edges in a $k$-out graph is at most $kn$, and the minimum degree is at least $k$.
Unfortunately, $k$-out graphs are not generally subgraphs of the random graph process
at the hitting time for minimum degree $k$.
We will therefore analyse a different model,
which can be considered the undirected counterpart of the $k$-out graph (see the next paragraph).
Nevertheless, we will exploit the simplicity of the $k$-out graph to analyse its (slightly more complicated)
undirected counterpart.

\paragraph{Random $k$-nearest neighbour graphs}
Suppose the edges of the complete graph are endowed with
(potentially random; distinct) ``lengths''.
The graph made up only of the $k$ shortest edges incident with each vertex
is called the {\em (random) $k$-nearest neighbour graph},
and has been studied in various, primarily geometric, contexts
(see, e.g.,~\cite{EPY97}).
When the weights are independent uniform random variables supported on $[0,1]$,
the model becomes a symmetric random graph\footnote{%
The choice of uniform distribution supported on $[0,1]$ is arbitrary;
any distribution without atoms would yield the same random graph.},
which we denote by~$\cO_k$ (see \cref{sec:ok}).
To prove \cref{thm:mindeg:hit}, we devise a simple strategy that emulates $\cO_k$.
Since the minimum degree of $\cO_k$ is obviously $k$,
the statement would follow from a theorem of Cooper and Frieze~\cite{CF95}
according to which the number of edges of $\cO_k$ is at most $cn$
for some concrete constant $c=c(k)<k$
which is roughly $k/2$ when $k$ is large
(see \cref{thm:ok:size,cor:ok:size}).

By observing that $\cO_k$ is stochastically dominated by the random $k$-out graph (\cref{obs:ok:kout}),
we prove that it is a connected (\cref{cor:ok:conn}) expander (\cref{cor:ok:exp}),
and use that to prove that it is Hamiltonian (\cref{thm:ok:ham,thm:ok:ham:hit}).
We use that helpful fact in several places:
when aiming to create a Hamilton cycle at the hitting time and with inflated budget,
we construct a subgraph of $\cO_k$ that is already Hamiltonian;
when creating a Hamilton cycle under an asymptotically optimal budget, we emulate copies of $\cO_k$
in small parts of the graph, to allow us to {\em absorb} a long path into a Hamilton cycle;
finally, we use a similar approach for the construction of perfect matchings.

\paragraph{High-level arguments for the containment of fixed graphs}
We give brief proof outlines for \cref{thm:trees,thm:cycles}.
The proof for the upper bound on the budget threshold for trees is essentially inductive:
given a fixed $k$-vertex tree $T$, we let $T'$ be obtained from $T$ by removing a leaf.
The inductive argument is that Builder can construct sufficiently many copies of $T'$
while leaving enough time to extend some of them into copies of $T$.
The proof for the lower bound is based on a similar idea:
we show that Builder cannot construct a connected component of size at least $k$.
In order to construct such a component, he needs to construct enough smaller components
so that in the remaining time, he will, \whp{}, connect two of them
so that the resulting component will be large enough.

The strategy for obtaining a cycle of length $\ell=2k+1$ goes through the construction of many ``traps'',
namely, non-edges whose addition to Builder's graph would create the desired cycle.
An optimised way to construct many such traps quickly is by constructing
$r$ $d$-ary trees of depth $k$ (for a correct choice of $r,d$).
If $d$ is large, most pairs of leaves in such a tree are connected by a path of length $2k$,
and thus form a trap.
The argument for a cycle of length $\ell=2k+2$ is similar.
We complement the upper bound with two lower bounds;
each matches the upper bound in a different regime.
We first show a ``universal'' lower bound, based on the straightforward observation that the number of traps
Builder has in his graph is bounded by $O(b^2)$,
as he has at most $2b$ vertices of positive degree in his graph.
A more involved lower bound fits the earlier regimes (lower values of $t$).
To this aim, we observe that Builder's graph is typically $\Theta(1)$-degenerate.
Then, we use this observation,
together with an estimate on the number of paths of a fixed length that contain a vertex
of ``high'' degree,
to bound from above the number of paths of length $\ell$ (and hence the number of traps).

\subsection{Related work}\label{sec:related}
The study of random graph processes is at the heart of the theory of random graphs,
providing a dynamic point of view on their evolution.
One of the main questions is to determine the thresholds of monotone increasing graph properties
or their (more refined) hitting times.
Classical results of this sort include the thresholds of minimum degree and vertex-connectivity $k$~\cite{ER61},
the appearance of a ``giant'' component~\cite{ER60},
Hamiltonicity~\cites{KS83,Bol84}
and the containment of fixed subgraphs~\cites{ER60,Bol81}.
For a comprehensive coverage of the topic, we refer the reader to
the books of Bollob\'as~\cite{Bol}, Janson, {\L}uczak and Ruci\'nski~\cite{JLR}, and
Frieze and Karo\'nski~\cite{FK}.
In our context,
an obvious necessary condition for the existence of a winning strategy for Builder
is that the ``time'' is above the threshold (or at least at the hitting time),
as otherwise, the underlying graph process is not guaranteed to have the desired property \whp{}.
Evidently, if $b=t$ (namely, if the budget equals the time),
and $t$ is (sufficiently) above the threshold of the target property,
then Builder has a winning strategy: the naive one that purchases each observed edge.
Since in many cases Builder can do better,
the model can be seen as an extension of the ``standard'' random graph process.

In the last couple of decades,
partly inspired by the remarkable work of Azar, Broder, Karlin and Upfal~\cite{ABKU99}
on {\em balanced allocations},
there has been a growing interest in {\em controlled} random processes.
In the context of graph processes,
an algorithm is provided with a random flow of edges (usually, but not always, the random graph process)
and with (the {\em offline} version)
or without (the {\em online} version)
``peeking into the future'', makes a decision that handles that flow
by accepting/rejecting edges, by colouring them, or by other means.
We mention several related models that fall into this category, mainly in the online version.

The {\em Achlioptas process},
proposed by Dimitris Achlioptas in 2000,
is perhaps the most studied controlled random graph process.
In this process, the algorithm is fed by a stream of random $k$-sets of edges (with or without repetitions)
and should pick, immediately and irrevocably (``online''), a single edge to accept, rejecting the others.
The algorithm's goal is to make the graph of accepted edges satisfy some monotone
increasing (decreasing) graph property
while minimising (maximising) the total number of rounds.
Early work on this model~\cites{BF01,BF01add,BFW04} treated the original question posed by Achlioptas
of {\em avoiding} a giant component for as long as possible.
Other works~\cites{FGS05,BK06} considered the opposite of the original question,
namely, the question of accelerating the appearance of the giant (see also \cite{SW07} for a general framework).
The model has also been studied for other objectives,
such as
avoiding~\cites{KLS09,MST11} or creating~\cites{KS12} a fixed subgraph
or obtaining a Hamilton cycle~\cite{KLS10}.
Several variants have also been studied, such as an {\em offline} version,
in which the algorithm sees all $k$-sets of edges at the beginning of the process
(see, e.g.,~\cite{BK06off});
and a {\em memoryless} version,
in which the algorithm's decision may not depend on its previous decisions
(see~\cite{BBFP09}).
A common generalisation of the Achlioptas process
and the process considered in this paper,
where an algorithm should pick {\em at most} a single edge to accept/purchase
subject to a restricted ``budget'',
was recently studied by Anastos in~\cite{Ana22+}
(see remarks after \cref{thm:ham,thm:ham:hit,thm:pm}).

The {\em semi-random process},
proposed by the third author in 2016,
is a variation of the Achlioptas process
in which
the algorithm is fed by a stream of random spanning stars (with repetitions)
instead of $k$-sets of edges.
As with the Achlioptas process, the algorithm must immediately and irrevocably pick
a single edge to accept.
Work on this model treats monotone properties such as
containment of fixed subgraphs~\cites{BHKPSS20,BMPR22+},
minimum degree and vertex-connectivity~\cites{BHKPSS20},
perfect matchings~\cite{GMP22+pm},
Hamilton cycles~\cites{GKMP22,GMP22+ham}
and bounded degree spanning graphs~\cite{BGHK20}.

Another model that fits in this setting was studied by Frieze and Pegden~\cites{FP18,FP18cor}
and by Anastos~\cite{Ana16+}
under the name ``purchasing under uncertainty''.
In their model, whenever an edge arrives, it is given an independent random ``cost'',
and the algorithm has to decide whether to purchase that edge,
aiming to pay the minimum total cost required to obtain a desired graph property.

In a Ramsey-type version of controlled random processes,
incoming edges are coloured (irrevocably) by an online algorithm.
The algorithm aims to avoid, or to create, a monochromatic property.
In~\cite{FKRRT03}, the triangle-avoidance game for up to $3$ colours is discussed.
This was extended to any fixed cyclic graph and any number of colours in~\cites{MSS09,MSS09ub}.
The model was further studied in the contexts of the giant component~\cites{BFKLS11,SST10}
and Hamilton cycles~\cite{BFKLS18}.

Finally, we would like to mention a related adaptation of the {\em two-stage optimisation with recourse} framework
for the {\em minimum spanning tree} model~\cite{FFK06}.
Here, every edge of the complete graph is given an independent random ``Monday'' cost
and another independent random ``Tuesday'' cost.
The algorithm sees all Monday costs and decides (immediately and irrevocably) which edges to incorporate into its (future) spanning tree.
Then, Tuesday costs are revealed, and the algorithm uses them to complete its (current) forest into a spanning tree.
The algorithm's goal is to minimise the total cost of edges in its constructed tree.

\subsection{Paper organisation and notation}

The rest of the paper is organised as follows.
In \cref{sec:prelim} we introduce some preliminaries.
\Cref{sec:ok} is devoted to the random $k$-nearest neighbours graph and
to the $k$-out graph.
\Cref{sec:span} contains the proofs of \cref{thm:mindeg:hit,thm:mindeg,thm:ham,thm:ham:hit,thm:pm},
and \cref{sec:small} contains the proofs of \cref{thm:trees,thm:cycles}.
In the final section, \cref{sec:conclude}, we mention a few relevant open problems.

Throughout the paper, all logarithms are in the natural basis, unless stated otherwise.
If $f,g$ are functions of $n$ we write $f\preceq g$ if $f=O(g)$,
$f\ll g$ if $f=o(g)$,
$f\asymp g$ if $f=\Theta(g)$,
and $f\sim g$ if $f=(1+o(1))g$.
For two vertices $u,v$ of a graph we write $u\sim v$ to denote that they are neighbours.
For simplicity and clarity of presentation,
we often make no particular effort to optimise the constants obtained in our proofs
and omit floor and ceiling signs whenever they are not crucial.

\section{Preliminaries}\label{sec:prelim}
\subsection{Concentration inequalities}

We will make use of the following version of Chernoff bounds (see, e.g., in,~\cite{JLR}*{Chapter 2}).
\begin{theorem}[Chernoff bounds]\label{thm:chernoff}
  Let $n\ge 1$ be an integer and let $p\in[0,1]$,
  let $x\sim\Bin(n,p)$, and let $\mu=\E{x}=np$.
  Then, for every $a>0$,
  \[
    \pr(x\le \mu-a) \le \exp\left(-\frac{a^2}{2\mu}\right),\qquad
    \pr(x\ge \mu+a) \le \exp\left(-\frac{a^2}{2(\mu+a/3)}\right).
  \]
\end{theorem}

The next proposition is often helpful in combination with \cref{thm:chernoff}.
\begin{proposition}\label{prop:chernoff}
  Let $n\ge 1$ be an integer and let $p\in[0,1]$,
  write $\mu=np$,
  and let $x_1,\dots,x_n$ be (not necessarily independent) Bernoulli random variables.
  Suppose further that for every $(y_1,\dots,y_n)\in\{0,1\}^n$
  and for every $i=1,\dots,n$,
  $\pr(x_i=1\mid x_j=y_j,\ \forall 1\le j<i)\ge p$.
  Then,
  $x=\sum_{i=1}^n x_i$ stochastically dominates a binomial random variable
  with $n$ attempts and success probability $p$.

  Similarly, if for every $(y_1,\dots,y_n)\in\{0,1\}^n$
  and for every $i=1,\dots,n$,
  $\pr(x_i=1\mid x_j=y_j,\ \forall 1\le j<i)\le p$,
  then
  $x$ is stochastically dominated by a binomial random variable
  with $n$ attempts and success probability $p$.
\end{proposition}

The following are trivial yet useful bounds.
\begin{claim}\label{cl:bin:lowtail}
  Let $n\ge 1$ be an integer and let $p\in[0,\frac{1}{2}]$,
  and let $x\sim\Bin(n,p)$.
  Write $q=1-p$ and let $1\le k\le np/q$ be an integer.
  Then
  \begin{equation*}
    \pr(x\le k) \le \left(\frac{enp}{kq}\right)^k e^{-np}.
  \end{equation*}
\end{claim}

\begin{proof}
  By the binomial theorem, for every $\alpha\in(0,p/q]$,
  \[ 
    (1+\alpha)^n
    = \sum_{i=0}^n \binom{n}{i}\alpha^i
    \ge \sum_{i=0}^k
    \binom{n}{i}
      \left(\frac{p}{q}\right)^i
      \left(\frac{\alpha q}{p}\right)^k,
  \]
  hence
  \[
    \pr(x\le k)
    = \sum_{i=0}^k \binom{n}{i}\left(\frac{p}{q}\right)^i q^n
    \le \frac{(1+\alpha)^n p^k}{\alpha^kq^k} \cdot (1-p)^n
    \le \frac{e^{\alpha n} p^k}{\alpha^kq^k} \cdot e^{-np}.
  \]
  Taking $\alpha=k/n\le p/q$ we obtain
  \[
    \pr(x\le k)
    \le \left(\frac{enp}{kq}\right)^k \cdot e^{-np}.\qedhere
  \]
\end{proof}

\subsection{The random graph process}\label{sec:gnm}
Recall that $G_t$ denotes the random graph process at time $t$.
The next (pretty standard) lemma essentially states that in $G_t$, where $t$ is superlinear (but not too large),
there are no small dense sets.
\begin{lemma}\label{lem:span}
  Suppose that $R\gg 1$ and $n^{-5}R^{2}t^3\le\alpha$ for some constant $\alpha<1/8$.
  Then,
  in $G_t$, \whp{}, every vertex set of size $r\le R$ spans at most $3r$ edges.
\end{lemma}

\begin{proof}
  Let $N=\binom{n}{2}$ and $p=t/N\sim 2tn^{-2}$.
  For a set $S$ of $s$ edges
  the probability that $S$ is contained in $G_t$
  is at most $p^s$.
  By the union bound, the probability that there exists a set $U$ of size $r$ that spans at least $3r$ edges is at most
  \begin{equation*}
    \binom{n}{r}\binom{\binom{r}{2}}{3r}p^{3r}
    \le \left(n r^2p^3\right)^r.
  \end{equation*}
  We note that $nR^2p^3\sim 8n^{-5}R^2t^3\le 8\alpha$,
  and hence $nRp^3\sim 8\alpha/R\ll 1$.
  By the union bound over $2\le r\le R$,
  the probability that there exists a set violating the promise of the lemma
  is at most
  \[
    \sum_{r=2}^R \left(n r^2p^3\right)^r
    \le \sum_{r=2}^{\sqrt{R}} \left(nRp^3\right)^r
    + \sum_{r=\sqrt{R}}^{R} \left(nR^2p^3\right)^r
    \le \sum_{r=2}^{\sqrt{R}} (o(1))^r
    + \sum_{r=\sqrt{R}}^R ((1+o(1))8\alpha)^r = o(1),
  \]
  (here the $o(1)$ terms in the penultimate expression are uniform in $r$)
  and the claim follows.
\end{proof}

The lemma easily implies the following useful fact.
\begin{claim}\label{cl:degen}
  Suppose $1\ll b^2\le n^5/(40t^3)$.
  Then, \whp{},
  every subgraph of $G_t$ with at most $b$ edges is $6$-degenerate.
\end{claim}

\begin{proof}
  Let $B$ be a subgraph of $G_t$ with at most $b$ edges.
  Thus, $B$ has at most $2b$ non-isolated vertices.
  Apply \cref{lem:span} with $R=2b$ and $t$
  (we can do so as $b\gg 1$ and $n^{-5}(2b)^2t^3\le 1/10$).
  We therefore obtain that, with high probability, for every $q\ge 0$,
  every subgraph of $B$ on $q$ vertices spans at most $3q$ edges,
  and hence has minimum degree of at most $6$,
  and the result follows.
\end{proof}

We introduce more notation.
For $d\ge 1$, let $X_d^t$ be the set of vertices in $G_t$ with degree less than $d$,
and let $Y_d^t$ be its complement.
The next lemmas state that in time linear in $n$ there is a transition
between ``most vertices are of degree less than $d$''
and ``most vertices are of degree at least $d$''.

\begin{lemma}\label{lem:gnt:large}
  Let $d\ge 1$ be an integer.
  Then, deterministically,
  $|Y_d^t|\le 2t/d$.
\end{lemma}

\begin{proof}
  It follows from $t=|E(G_t)|\ge |Y_d^t|\cdot d/2$.
\end{proof}

\begin{lemma}\label{lem:gnt:small}
  For every integer $d\ge 1$,
  if $t\ge 6dn$
  then $|X_d^t|\le n/100$ \whp{}.
\end{lemma}

\begin{remark}
  When proving statements about $G_t$
  it is often convenient to prove them for $G\sim G(n,p)$ with $p\sim t/\binom{n}{2}$ instead.
  We could then translate the results back to $G_t$ using standard methods
  (see, e.g., in~\cite{FK}).
\end{remark}

\begin{proof}
  Let $t=6dn$ 
  (the result would follow for $t\ge 6dn$ due to monotonicity).
  Write $p=12d/n\sim t/\binom{n}{2}$
  and $m=n/100$.
  If $|X_d^t|>m$ then there exists a vertex set $V_0$ with $|V_0|= m$
  such that $|E(V_0,\overline{V_0})|<md$.
  Let $x=|E(V_0,\overline{V_0})|$ and note that $x\sim\Bin(m(n-m),p)$.
  By \cref{cl:bin:lowtail} and the union bound over the choice of $V_0$,
  the probability that the statement of the lemma does not hold is at most
  \begin{align*}
    \binom{n}{m}\pr(x<md)
    &\le \left(\frac{en}{m}\right)^m \left(\frac{em(n-m)p}{md(1-p)}\right)^{md} e^{-m(n-m)p}\\
    &\le \left(\frac{en}{m}\right)^m \left(\frac{enp}{d}\right)^{md} e^{-0.99mnp}
    \le \left(100e\cdot\left(12e\cdot e^{-11}\right)^d\right)^m  = o(1),
  \end{align*}
  where in the last computation we used the fact that $e^{10}>1200$.
\end{proof}

Recall that for every positive integer $k$,
$\tau_k$ is the hitting time for minimum degree $k$ in the random graph process.
We will use the next two known results about $\tau_k$.

\begin{lemma}[\cite{ER61}]\label{lem:gnt:deg:hit}
  For every fixed $k\ge 1$,
  $\tau_k = n(\log{n}+(k-1)\log\log{n}+x)/2$,
  where $x=O(1)$ \whp{}.
\end{lemma}

For a proof of the next lemma see, e.g.,~\cite{Kri16}.
\begin{lemma}\label{lem:gnt:small:far}
  Let $k\ge 1$ be an integer.
  Then, \whp{}, $|X^{\tau_2}_k|\le n^{0.5}$,
  the distance between two vertices in $X^{\tau_2}_k$ is at least $5$,
  and no vertex in $X^{\tau_2}_k$ lies in a cycle of length at most $4$.
\end{lemma}

Finally, we would need the following lemma on the size of the $k$-core of a random graph.
The \defn{$k$-core} of a graph is its (unique) maximal subgraph of minimum
degree at least $k$.
\begin{lemma}\label{lem:gnt:core}
  For every $k\ge 1$ the size of the $k$-core of $G_t$, for $t=8kn$, is \whp{} at least $n/2$.
\end{lemma}

\begin{proof}
  We prove the statement for $G(n,p)$ with $p=16 k/n\sim t/\binom{n}{2}$.
  First we show that \whp{} the number of edges between any set of size $n/2$ and its complement
  is at least $kn$.
  Indeed, let $V_0$ be a fixed vertex set of size $n/2$.
  Then, $x=|E(V_0,\overline{V_0})|\sim\Bin\left(n^2/4,p\right)$,
  hence $\E{x}=4kn$.
  By Chernoff bounds (\cref{thm:chernoff}),
  $\pr(x\le kn) \le \exp(-\frac{9}{8}kn)$.
  By the union bound over all choices of $V_0$, the probability that there exists such a set is at most
  $2^n\cdot\exp(-\frac{9}{8}kn)=o(1)$.

  Now, suppose the $k$-core is smaller than $n/2$.
  Consider the process of removing vertices of degree smaller than $k$ one by one.
  By definition of the $k$-core, this process lasts more than $n/2$ steps.
  Stop the process exactly after $n/2$ steps.
  At this point,
  the remaining set $|V_0|$ is of size $n/2$.
  On the other hand, we have removed less than $k\cdot n/2$ edges,
  hence $|E(V_0,\overline{V_0})|<kn$,
  and we have seen that \whp{} there is no such set.
\end{proof}

\subsection{Rotations, expanders, and boosters}\label{sec:boost}
In our proofs we shall repeatedly use the so-called
{\em rotation--extension} technique of P\'osa~\cite{Pos76}.
Given a longest path $P=(v_0,\ldots,v_j)$
we say that $P'$ obtained from $P$ by an \defn{elementary rotation} of $P$ (with $v_0$ fixed)
if $P'=(v_0,\ldots,v_{i-1},v_i,v_j,v_{j-1},\ldots,v_{i+1})$ ($i<j$).
We let $\cR(P)$ denote the set of endpoints of paths
obtained from $P$ by a (finite) sequence of elementary rotations.
For a vertex set $A$, denote by $N(A)$ the {\em external} neighbourhood of $A$,
that is, the set of all vertices which are not in $A$ that have a neighbour in $A$.
We will use the following classical lemma of P\'osa:
\begin{lemma}[P\'{o}sa's lemma~\cite{Pos76}]\label{lem:Posa}
  Let $G$ be a graph and let $P$ be a longest path in $G$.
  Then $|N(\cR(P))| \leq 2|\cR(P)|-1$.
\end{lemma}

Say that a graph $G=(V,E)$
is an \defn{$R$-expander} if every set $U\subseteq V$ with $|U|\le R$ has $|N(U)|\ge 2|U|$.
For an $n$-vertex graph $G$ denote by $\lambda(G)$ the length of a longest path in $G$,
or $n$ if $G$ is Hamiltonian.
A non-edge $e$ of $G$ is called a \defn{booster}
if $\lambda(G+e)\ge\min\{\lambda(G)+1,n\}$.
The following lemma is a corollary of \cref{lem:Posa}.
For a proof see, e.g.,~\cite{Kri16}.
\begin{lemma}\label{lem:boosters}
  Let $G$ be a connected $R$-expander which contains no Hamilton cycle.
  Then $G$ has at least $(R+1)^2/2$ boosters. 
\end{lemma}

In our arguments below,
we combine \cref{lem:boosters} with the {\em sprinkling} technique,
standard for many Hamiltonicity arguments in random graphs.
We begin with an
expander, and add random edges.
Typically, the first few random edges turn it into a connected graph
(if it was not connected to begin with),
and then the remaining random edges hit boosters
(which are numerous, according to \cref{lem:boosters}),
advancing the expander towards Hamiltonicity.
We summarize this in the following lemma:
\begin{lemma}[Sprinkling]\label{lem:sprinkling}
  For every $\beta>0$ there exists $C>0$ such that the following holds.
  For every $n$-vertex $\beta n$-expander,
  the addition of $Cn$ random non-edges to $G$
  makes it \whp{} Hamiltonian.
\end{lemma}

\begin{proof}
  Since $G$ is a $\beta n$-expander,
  each of its connected components has size at least $3\beta n$,
  so there are at most $1/(3\beta) = O(1)$ components.
  Thus, in the first $t_1\to\infty$ rounds,
  there will be, \whp{}, an edge (observed and purchased)
  between every pair of components.
  Thus, \whp{}, $G_{t_1}$ is connected.

  Now, using \cref{lem:boosters} and the monotonicity of expansion,
  we know that $G_{t_1}$ is either Hamiltonian (in which case we are done)
  or it has at least $\beta^2\binom{n}{2}$ boosters. 
  Thus, the probability that a random edge hits a booster is at least $\beta^2$.
  Let $t_2=\tau n$ be an integer with $\tau>\beta^{-2}$,
  and consider the next $t_2$ random edges.
  For $i\in\{t_1+1,\dots,t_1+t_2\}$ let $y_i$ be the indicator of the event that the $i$'th edge hits a booster,
  or that $G_{i-1}$ is already Hamiltonian.
  Let $y=\sum_{i=1}^{t_1} y_i$,
  and note that $y\ge n$ implies that $G_{t_1+t_2}$ is Hamiltonian.
  Since $y$ stochastically dominates a binomial random variable
  with $t_2$ attempts and success probability $\beta^2$,
  it follows from \cref{thm:chernoff}
  that the probability that $y\le n$ is $o(1)$.
  Thus, after adding $t_2$ further random non-edges, the graph becomes \whp{} Hamiltonian.
\end{proof}

We also use \cref{lem:Posa} to show that in Hamiltonian expanders,
the set of endpoints of Hamilton paths whose other endpoint is fixed
is large.
\begin{lemma}\label{lem:manyends}
  Let $G$ be a Hamiltonian $R$-expander and let $v$ be a vertex of $G$.
  Then, the number of endpoints of Hamilton paths of $G$ whose other endpoint is $v$
  is at least $R$.
\end{lemma}

\begin{proof}
  Since $G$ is Hamiltonian,
  there exists a Hamilton path $P$ with $v$ as an endpoint.
  By \cref{lem:Posa},
  $|N(\cR(P))|\le 2|\cR(P)|-1$;
  thus $|\cR(P)|>R$.
\end{proof}

\subsection{\texorpdfstring{Greedy $k$-matchings and $k$-cores}{Greedy k-matchings and k-cores}}
\label{sec:matching}
We show that the greedy strategy works ``well'' for purchasing a large $k$-matching;
namely, a subgraph of maximum degree $k$ in which all but a few vertices are of degree $k$.
This will be useful in the proofs of \cref{thm:mindeg} and \cref{thm:pm}.

\begin{lemma}\label{lem:kmatching}
  Let $k$ be a positive integer and let $\eps>0$.
  Then, there exists $C=C(k,\eps)>0$ such that if $t\ge Cn$ and $b\ge (k-\eps)n/2$ then
  there exists a $(t,b)$-strategy of Builder
  (that succeeds \whp{})
  that purchases a graph with maximum degree $k$
  in which all but at most $\eps n$ vertices are of degree $k$.
\end{lemma}

\begin{proof}
  Builder follows the greedy strategy;
  that is, he purchases every edge both of whose endpoints are of degree below $k$.
  Observe that this strategy ensures that the maximum degree of Builder's graph is at most $k$ at any stage.
  Let $U$ denote the (dynamic) set of vertices of degree below $k$.
  Let $C=k\eps^{-2}$ and let $t\ge Cn$.
  For $i=1,\ldots,t$ let $x_i$ be the indicator of the event
  that the $i$'th edge arriving is contained in $U$
  (and thus purchased by Builder) {\em or} $|U|\le\eps n$.
  The probability for this event is at least $\sim\eps^2$.
  Thus, by \cref{prop:chernoff},
  $x:=\sum x_i$ stochastically dominates a binomial random variable with mean $\sim \eps^2Cn=kn$.
  Therefore, the probability that $x < (k-\eps)n/2$ is,
  by Chernoff bounds (\cref{thm:chernoff}), $o(1)$.
  On the event that $x\ge (k-\eps)n/2$,
  either $|U|\le \eps n$, or Builder has purchased at least $(k-\eps)n/2$ edges,
  in which case it follows that $|U|\le \eps n$.
  Obviously, by that time Builder has purchased at most $(k-\eps)n/2$ edges.
\end{proof}

Given a vertex set we may use \cref{thm:ok:size} below to construct a spanning graph of an arbitrary fixed minimum degree
(see \cref{sec:mindeg}).
That requires, however, that the time will be long enough so that every vertex will have the desired degree
in the underlying graph process.
If we are satisfied with a {\em small} graph of a given minimum degree, we can do it in linear time.
The following corollary follows directly from
\cref{lem:gnt:core}.
\begin{corollary}\label{cor:core}
  Let $k\ge 1$ be an integer
  and let $b=t=8kn$.
  Then,
  by purchasing every observed edge,
  Builder has a $(t,b)$-strategy
  resulting (\whp{}) in a graph
  whose $k$-core is of size at least $n/2$.
\end{corollary}

\section{\texorpdfstring{The random $k$-nearest neighbour graph}{The random k-nearest neighbour graph}}
\label{sec:ok}
Suppose the edges of the complete graph are endowed with independent uniform random ``lengths'' in $[0,1]$.
The (random) graph obtained by retaining (only) the $k$ shortest edges incident with each vertex is called
the \defn{random $k$-nearest neighbour graph} (see~\cites{CF95,FK}).
We denote it by $\cO_k$ (or by $\cO_k(n)$ if we wish to emphasize that it is on $n$ vertices).
As we remarked earlier,
the choice of uniform distribution supported on $[0,1]$ is arbitrary;
any distribution without atoms would yield the same random graph (see~\cite{CF95});
we use this fact in the proof of \cref{obs:ok:kout}.
The following theorem of Cooper and Frieze~\cite{CF95} (stated here in a weaker form) will be useful for us
in several places in this paper.
\begin{theorem}[\cite{CF95}*{Theorem 1.1}]\label{thm:ok:size}
  For every constant $k$,
  the number of edges of $\cO_k$ is \whp{}
  \[
    kn - \frac{n(n-1)}{2n-3}\sum_{1\le i\le j\le k}
    \frac{\binom{n-2}{i-1}\binom{n-2}{j-1}}{2^{\delta(i,j)}\binom{2n-4}{i+j-2}}
    + O(\sqrt{n}\log{n}),
  \]
  where $\delta(i,j)$ is the Kronecker delta.
\end{theorem}

It would be useful for us to state \cref{thm:ok:size} in an ``asymptotic'' form.
To this end,
we complement the analysis of Cooper and Frieze
by obtaining an asymptotic estimate of the expected size of $\cO_k$.
Let
\begin{equation}\label{eq:ok}
  \co_k = k-\frac{1}{4}\sum_{i,j=0}^{k-1}\binom{i+j}{i}2^{-(i+j)}.
\end{equation}

\begin{corollary}\label{cor:ok:size}
  For every integer $k\ge 1$
  the number of edges of $\cO_k$ is \whp{}
  $\co_k n + O(\sqrt{n}\log{n})$.
  Moreover,
  $k/2<\co_k\le 3k/4$,
  and $\lim_{k\to\infty} \co_k/k = 1/2$.
\end{corollary}

\begin{proof}
  First note that 
  for every nonnegative integer $k$,
  \begin{equation}\label{eq:coins}
    \sum_{i=0}^k \binom{i+k}{i}2^{-(i+k)} = 1.
  \end{equation}
  Indeed,
  consider a sequence of fair coin flips that stops whenever $k+1$ heads {\em or}
  $k+1$ tails are encountered.
  Let $x$ be the number of tails if the sequence stopped by encountering $k+1$ heads,
  or the the number heads otherwise.
  Thus, $x$ is supported on $\{0,\dots,k\}$
  and $\pr(x=i) = \binom{i+k}{i}2^{-(i+k)}$.
  For an integer $k\ge 0$ write $f(k)=\sum_{i,j=0}^k\binom{i+j}{i}2^{-(i+j)}$
  (so $\co_k=k-f(k-1)/4$).
  Note that for $k\ge 1$, by \cref{eq:coins},
  \[
    f(k) = f(k-1) + 2\sum_{i=0}^k\binom{i+k}{i}2^{-(i+k)} - \binom{2k}{k}2^{-2k}
         = f(k-1) + 2 - \Theta(k^{-1/2}).
  \]
  Thus, $f(0)=1$, $f(k)\le 1+2k$ for $k\ge 0$ and
  \begin{equation}\label{eq:coins:lim}
    \lim_{k\to\infty} \frac{f(k)}{k} = 2.
  \end{equation}
  It follows from \cref{thm:ok:size} that
  the number of edges of $\cO_k$ is, \whp{},
  \[\begin{aligned}
    & \left(k - \frac{1}{2} \sum_{1\le i\le j\le k} \binom{i+j-2}{i-1}2^{-(i+j-2)-\delta(i,j)}\right)n
      + O(\sqrt{n}\log{n})\\
    &= \left(k - \frac{1}{4} \sum_{i,j=0}^{k-1} \binom{i+j}{i} 2^{-(i+j)}\right)n
      + O(\sqrt{n}\log{n})
    = \co_kn + O(\sqrt{n}\log{n}).
  \end{aligned}\]
  Evidently, $f(k)$ is increasing (hence $\co_k$ is decreasing);
  thus,
  $\frac{k}{2}+\frac{1}{4} \le \co_k\le \co_1= 3k/4$.
  In addition, \cref{eq:coins:lim} implies that $\co_k/k\to1/2$ as $k\to\infty$.
\end{proof}

One of the main results of~\cite{CF95},
that we conveniently use here,
is that for $k\ge 3$, $\cO_k$ is \whp{} $k$-vertex-connected.
Cooper and Frieze also showed (there) that
the probability that $\cO_2$ is connected is bounded away from $0$ and $1$.
\begin{theorem}[\cite{CF95}*{Theorem 1.4}]\label{thm:ok:conn}
  For $k\ge 3$, $\cO_k$ is \whp{} $k$-vertex-connected.
\end{theorem}

Our main goal in this section is to prove that for large enough $k$,
$\cO_k$ is \whp{} Hamiltonian.

\begin{theorem}\label{thm:ok:ham}
  There exists $k_H>0$
  such that for every $k\ge k_H$,
  $\cO_k$ is \whp{} Hamiltonian.
\end{theorem}

As mentioned in the introduction, this partially resolves
\cite{FriBib}*{Problem 45}.
It is conjectured that $k_H=3$ or $k_H=4$.

In fact, we need (and hence prove) a stronger result that immediately implies \cref{thm:ok:ham}.
To state it,
consider the following equivalent way of generating $\cO_k$ (see~\cite{CF95}).
Given a uniform random ordering of the edges of the complete graph $K_n$,
we consider them one by one.
When an edge arrives, we add it to $\cO_k$ if and only  if one of its endpoints is incident
to a vertex that is --- at this point --- of degree less than $k$.
When we only add edges that arrived by time $t$ for some fixed $0\le t\le \binom{n}{2}$,
we obtain a subgraph of $\cO_k$ that we denote by $\cO_k^t$.
This process yields a coupling between the random graph process at time $t$
and $\cO_k^t$.
Under this coupling, which we call the \defn{natural coupling},
$\cO_k^t$ is a subgraph of $G_t$.
For convenience, we denote $\cO_{k,d}=\cO_k^{\tau_d}$ for every $0\le d\le k$.
Note that as long as the minimum degree of the random graph process is at most $1$,
it cannot contain a Hamilton cycle;
thus, for $t<\tau_2$, $\cO_k^t$ is not Hamiltonian.

\begin{theorem}\label{thm:ok:ham:hit}
  There exists $k_H>0$
  such that for every $k\ge k_H$,
  $\cO_{k,2}$ is \whp{} Hamiltonian.
\end{theorem}
As mentioned in the introduction,
a classical result due to Bollob\'as~\cite{Bol84}
and independently Ajtai, Koml\'os, and Szemer\'edi~\cite{AKS85}
is a hitting-time result, which states that, \whp{},
the random graph process will contain a Hamiltonian cycle as soon as its minimum degree reaches $2$.
\Cref{thm:ok:ham:hit} refines that result by showing that
there exists an {\em online} algorithm
that can choose a small number of edges (linear in the number of vertices)
from the first $\tau_2$ edges of the random graph process
that induce a Hamiltonian cycle \whp{}.
Furthermore, this algorithm is deterministic,
simple,
and only requires local information,
namely, the degrees of the endpoints of the considered edge.

\medskip

The rest of this section is devoted to the proof of \cref{thm:ok:ham:hit},
which implies, in turn, \cref{thm:ham:hit} (see \cref{sec:ham:hit}).
While the proof shares some ideas with previous proofs
of the hitting time result for Hamiltonicity of the random graph process
(see, e.g.,~\cite{Kri16}),
it still has its twists and turns to achieve the goal.
In particular,
we find a way around the common method of {\em random sparsification} accompanied by {\em sprinkling}
(which is not available to us in this setting)
by casting our argument directly on the (already sparse) graph
and finding boosters inside it using a sophisticated argument.
We begin with a couple of helpful observations.
The following is immediate:

\begin{observation}\label{obs:ok:deg}
  For integers $0\le d < k$,
  time $t$
  and a vertex $v$,
  under the natural coupling between $G_t$ and $\cO_k^t$ one has
  $d_{\cO_k^t}(v)=d$ if and only if $d_{G_t}(v)=d$.
\end{observation}

The next observation that we will need for this goal is that $\cO_k$ is stochastically dominated by the \defn{random $k$-out graph}
(see, e.g., in~\cite{FK}):
the random graph whose edges are generated independently for each vertex by a random choice of $k$ distinct edges incident to that vertex.
For brevity, denote it by $\cG_k$.
Note that $\cG_k$ has roughly $kn$ edges,
whereas, as stated in \cref{cor:ok:size},
$\cO_k$ is nearly twice as sparse (for large $k$).

\begin{observation}\label{obs:ok:kout}
  $\cO_k$ is stochastically dominated by $\cG_k$.
\end{observation}

\begin{proof}
  For every ordered pair of distinct vertices $u,v$ let
  $\w(u,v)$ be a uniform $[0,1]$ random weight assigned to the ordered pair $(u,v)$.
  Then, $\cG_k$ is obtained by adding the edge $\{u,v\}$
  whenever $\w(u,v)$ is one of the $k$ smallest weights in the family $\{\w(u,w)\}_{w}$.
  For every (unordered) pair of distinct vertices $u,v$ let
  $\x\{u,v\}=\min\{\w(u,v),\w(v,u)\}$.
  Observe crucially that while $\x$ is not uniform,
  it has a continuous density function on $[0,1]$
  and, more importantly, is independent for distinct edges.
  Thus, $\cO_k$ is obtained by including the edge $\{u,v\}$ whenever $\x\{u,v\}$ is one of the $k$ smallest weights in the family
  $\{\x\{u,w\}\}_w$.
  Let us now show that under this coupling, $\cO_k\subseteq\cG_k$.
  Let $\{u,v\}\in\cO_k$ and assume without loss of generality that $\x\{u,v\}=\w(u,v)$.
  Pick a vertex $w$ for which $\x\{u,v\}<\x\{u,w\}$ (there are at least $n-1-k$ such vertices).
  Then, $\w(u,v) = \x\{u,v\} < \x\{u,w\} \le \w(u,w)$.
  It follows that there are at least $n-1-k$ vertices $w$ for which $\w(u,v)<\w(u,w)$,
  hence $\{u,v\}\in\cG_k$.
\end{proof}

We now show that for $k\ge 12$, $\cO_{k,2}$ is \whp{} an expander\footnote{%
We made no effort to optimise the constant $12$ in the lemma.}.

\begin{lemma}\label{lem:ok:exp}
  There exists $\beta>0$
  such that
  for $k\ge 12$,
  $\cO_{k,2}$ is \whp{} a $\beta n$-expander.
\end{lemma}

\begin{proof}
  We prove it for $k=12$ and derive the statement for every $k\ge 12$ due to monotonicity.
  Fix $b=k/11>1$.
  We first prove that there exists $\gamma>0$ for which,
  \whp{}, no vertex set $S$ in $\cO_{k,2}$ of cardinality $s\le\gamma n$
  spans more than $bs$ edges.
  Since $\cO_{k,2}$ is stochastically dominated by $\cO_k$
  and (by \cref{obs:ok:kout}) $\cO_k$ is stochastically dominated by $\cG_k$,
  we may prove that statement for $\cG_k$.
  Note that the probability that a given edge appears in $\cG_k$
  is at most twice the probability that a given vertex chooses a given incident edge
  (which is $k/(n-1)$).
  By the union bound over the choice of $S$ and the choices of the spanned edges,
  the probability that there exists a vertex set $S$ in $\cG_{k}$ of cardinality $s\le\gamma n$
  that spans more than $bs$ edges is at most
  \[
    \sum_{s=2}^{\gamma n} \binom{n}{s}\binom{\binom{s}{2}}{bs}
    \left(\frac{3k}{n}\right)^{bs}
    \le \sum_{s=2}^{\gamma n}
    \left[
      \frac{en}{s} \left(\frac{es}{2b}\cdot\frac{3k}{n}\right)^b
    \right]^s
    = \sum_{s=2}^{\gamma n}\left(\Theta\left(\left(\frac{s}{n}\right)^{b-1}\right)\right)^s.
  \]
  If $n/\log{n}\le s\le\gamma n$ then the $s$'th summand is at most $c^{n/\log{n}}$ for some $c<1$
  (for small enough $\gamma>0$).
  Otherwise, it is $(o(1))^s$.
  Thus, the sum is $o(1)$.
  
  Back to $\cO_{k,2}$,
  let $X=X^{\tau_2}_k$ be the set of vertices of degree smaller than $k$
  at time $\tau_2$.
  Suppose first that there exists a set $B$ of size at most $\beta n$ for $\beta=\gamma/5$
  with $B\cap X=\es$ and $|N(B)|\le 4|B|$.
  Set $S=B\cup N(B)$ and note that $s=|S|\le 5\beta n=\gamma n$
  while $|E(S)|\ge k\beta n/2>bs$, contradicting the above.
  Now, let $A$ be a vertex set with $|A|\le \beta n$.
  Let $A_X=A\cap X$ and $A_Y=A\sm X$.
  Let further $S_X=A_X\cup N(A_X)$ and $S_Y=A_Y\cup N(A_Y)$.
  By the above, $|N(A_Y)|\ge 4|A_Y|$.
  By \cref{lem:gnt:small:far,obs:ok:deg},
  $|N(A_X)|\ge 2|A_X|$
  and every vertex in $A_Y$ (or in general) has at most one neighbour in $S_X$.
  Hence, in particular, $|N(A_Y)\sm S_X|\ge |N(A_Y)|-|A_Y|\ge 3|A_Y|$.
  Thus,
  \[
    |N(A)|
    = |N(A_X)\sm A_Y|+|N(A_Y)\sm S_X|
    \ge (2|A_X|-|A_Y|) + 3|A_Y|
    \ge 2|A|,
  \]
  as required.
\end{proof}

Our next step is to show that $\cO_{k,2}$ is also typically connected
(for large enough $k$).
For that, we need the following lemma.
As we mentioned earlier, Cooper and Frieze~\cite{CF95} showed that
$\cO_2$ is  connected with probability that is bounded away from $0$ and $1$.
Since $\cO_3\not\subseteq\cO_{k,2}$ for any $k$,
we cannot use \cref{thm:ok:conn}. 

\begin{lemma}\label{lem:ok:quad}
  Let $\alpha>0$
  and let $F$ be a fixed set of edges in $K_n$ with $|F|\ge\alpha n^2$.
  Then, there exist $k_0=k_0(\alpha)$, $\eta=\eta(\alpha)$ and $c=c(\alpha)>0$
  such that for every $k\ge k_0$,
  if $t=\eta kn$
  then
  $\pr(F\cap E(\cO_k^{t})=\es)\le \exp(-ckn)$.
\end{lemma}

\begin{proof}
  Let $\gamma=\sqrt{\alpha}$,
  and $\eta=\gamma/3$.
  Take $k>6(1-\log{\gamma})/\gamma^2$,
  and $t_0=\eta kn$.
  Let $E_0=\{e_1,\ldots,e_{t_0}\}$ be the first $t_0$ edges of the underlying random graph process
  (in this order)
  and let $F_0=F\cap E_0$.
  Let $V_0=\bigcup F_0$ be the set of vertices covered by the edges of $F_0$.
  We observe that typically $V_0$ cannot be too small;
  indeed,
  for a given set $S$ to be $\bigcup F_0$
  we need every edge of $F$ that is not contained in $S$
  to appear after time $t_0$.
  By the union bound over all choices of $V_0$ of size at most $\gamma n$,
  \[
    \begin{aligned}
      \pr(|V_0|\le\gamma n)
      &\le \binom{n}{\gamma n} \left(1-\frac{t_0}{\binom{n}{2}}\right)^{|F|-\binom{\gamma n}{2}}
      \le \exp((\gamma-\gamma\log\gamma-\eta k(2\alpha-\gamma^2))n)\\
      &\le \exp((\gamma-\gamma\log{\gamma}-\gamma^3k/3)n)
      \le \exp(-ckn),
    \end{aligned}
  \]
  for $c=\gamma^3/6>0$.
  Let $F_1=F_0\cap E(\cO_k^{t_0})$.
  Note that if $F_1=\es$
  then when an edge in $F_0$ arrived both its endpoints had degree at least $k$;
  hence all vertices of $V_0$ have degree at least $k$ in $\cO_k^{t_0}$.
  But in this case, $|V_0|\ge\gamma n$ implies that
  $\gamma n k/2\le |E(\cO_k^{t_0})|\le t_0=\eta kn=\gamma kn/3$,
  a contradiction.
  Thus,
  \[
    \pr(F\cap E(\cO_k^t)=\es)
    \le \pr(F\cap E(\cO_k^{t_0})=\es)
    \le \pr(F_1=\es)
    = \exp(-ckn).
    \qedhere
  \]
\end{proof}

\begin{lemma}\label{lem:ok:large}
  For every $\beta>0$
  there exists $k_0=k_0(\beta)$
  such that for every $k\ge k_0$,
  \whp{},
  every two disjoint vertex sets in $\cO_{k,2}$ with $|A|,|B|\ge\beta n$
  are connected by an edge.
\end{lemma}

\begin{proof}
  Let $\alpha=\beta^2$
  and let $k_0,\eta,c$ be the constants guaranteed by \cref{lem:ok:quad}.
  Let $\ent(\beta)=-\beta\log{\beta}-(1-\beta)\log(1-\beta)$,
  and set $k_1=3\ent(\beta)/c$.
  Take $k\ge \max\{k_0,k_1\}$ and $t=\ceil{\eta kn}$.
  Fix two disjoint vertex sets $A,B$ with $|A|,|B|=\beta n$.
  Let $F=E(A,B)$, so $|F|\ge\alpha n^2$.
  By \cref{lem:ok:quad},
  the probability that in $\cO_k^t$ does not contain an edge from $F$
  (and hence the probability that in $\cO_k^t$ the sets $A,B$ are not connected by an edge)
  is at most $\exp(-ckn)$.
  By the union bound over all sets $A,B$,
  the probability that in $\cO_k^t$ there are two disjoint sets of size at least $\beta n$
  that are not connected by an edge is at most
  \[
    \binom{n}{\beta n}^2 e^{-ckn}
    = \exp(((2+o(1))H(\beta)-ck)n) = \exp(-\Omega(n)) = o(1).
  \]
  The result follows since \whp{} $t\le\tau_2$, implying $\cO_k^t\subseteq\cO_{k,2}$.
\end{proof}

\begin{corollary}\label{cor:ok:exp}
  There exists $k_0>0$
  such that for every $k\ge k_0$,
  $\cO_{k,2}$ is \whp{} an $\frac{n}{4}$-expander.
\end{corollary}

\begin{proof}
  Let $\beta$ be the constant guaranteed by \cref{lem:ok:exp}
  (we may (and will) assume that $\beta<1/4$),
  and let $k_0=k_0(\beta)$ be the constant guaranteed by \cref{lem:ok:large}.
  Let $k\ge\max\{k_0,12\}$.
  Thus, by \cref{lem:ok:exp}, $\cO_{k,2}$ is \whp{} a $\beta n$-expander;
  and, by \cref{lem:ok:large}, $\cO_{k,2}$ has \whp{} the property
  that any two disjoint vertex sets of size at least $\beta n$ are connected by an edge.
  Assume both properties hold
  and
  let $A$ be a vertex set with $|A|\le n/4$.
  If $|A|\le\beta n$ then $|N(A)|\ge 2|A|$ since $\cO_{k,2}$ is a $\beta n$-expander.
  Otherwise, there is an edge between $A$ and any other set of size at least $\beta n$,
  hence
  $|N(A)|\ge n-\beta n-|A|\ge (3/4-\beta)n\ge n/2 \ge 2|A|$.
\end{proof}

\begin{corollary}\label{cor:ok:conn}
  With the same $k_0>0$ as in \cref{cor:ok:exp},
  for every $k\ge k_0$,
  $\cO_{k,2}$ is \whp{} connected.
\end{corollary}

\begin{proof}
  Let $k_0$ be the constant guaranteed by \cref{cor:ok:exp}
  and let $k\ge k_0$.
  From \cref{cor:ok:exp} it follows that
  every connected component of $\cO_{k,2}$ is, \whp{}, of size at least $3n/4$,
  implying that the graph is connected.
\end{proof}

We note that \cref{cor:ok:conn}
can be obtained directly from \cref{lem:ok:exp,lem:ok:large}.
However, we state \cref{cor:ok:exp} separately,
as it will be used in the proof of \cref{thm:ok:ham:hit}.

\begin{proof}[Proof of \cref{thm:ok:ham:hit}]
  Let $k_0$ be the
  constant from \cref{cor:ok:exp},
  let $\ell=\max\{k_0,500\}$
  and let $k = 100\ell$
  (this $k$ will be the $k_H$ from the statement of the theorem).
  Note that by proving the statement of the theorem for $k$,
  due to monotonicity of $\cO_{k,2}$ in the first parameter,
  we prove it for every $k'\ge k$.
  Let $t_0=6\ell n$ and write $a=6\ell/k$, so $t_0=akn$.
  Consider the first $t_0$ edges of the random graph process
  (and recall from \cref{lem:gnt:deg:hit} that $t_0\ll\tau_2$ \whp{}).
  When an edge $e$ arrives
  we colour it {\em red} if it is incident to a vertex of degree $d<\ell$,
  or {\em black} otherwise.

  Recall (from \cref{sec:gnm})
  that $Y_\ell^t$ is the set of vertices in $G_t$ with degree at least $\ell$.
  Let $L=Y_\ell^{t_0}$,
  and note that by \cref{lem:gnt:small},
  $|L|\ge 0.99n$ \whp{}. 
  We now consider the next edges of the random graph process, until time $\tau_2$.
  When an edge $e$ arrives
  we colour it {\em red} if it is incident to a vertex of degree $d<\ell$
  (note that these red edges are {\em not} contained in $L$),
  or {\em black} if it is any other edge not contained in $L$.
  If the arriving edge is contained in $L$,
  we do not record its identity,
  but rather its arrival time.
  In particular, we count these edges.
  We observe crucially that the exact identity of such an edge
  does not affect the rest of the colouring process;
  indeed, any permutation of the edges of the random graph process
  that fixes the coloured edges (by time $\tau_2$) and the edges not in $L$
  is feasible.
  Let $m$ be the number of edges we counted in $L$
  and let $t_1,\dots,t_m$ be their arrival times
  ($t_0<t_1<\dots<t_m<\tau_2$).
  Note that the probability of an edge to fall inside $L$
  before $O(n)$ edges fell there (without conditioning on the rest of the process)
  is at least $p:=0.98$
  (since the number of existing edges in $L$ is $o(n^2)$).
  Thus, for any $b=O(1)$,
  the number of edges falling into $L$
  between $t_0$ and $t_0+bn$
  stochastically dominates a binomial random variable with $bn$ attempts
  and success probability $p$.
  Take $b=500$, and $m'=\floor{bnp/2}$.
  It follows that $m'<m$
  and $t_i<t_0+bn$ for every $1\le i\le m'$ \whp{}.

  Let $F$ be the set of non-edges in $L$ at time $\tau_2$.
  Conditioning on the entire colouring process
  (including $m$ and the times $t_1,\ldots,t_m$),
  the set of edges that fall into $L$ at times $t_1,\dots,t_m$
  is distributed as a uniformly sampled edge set of size $m$ of $F$.
  Let $f_1,\ldots,f_m$ be a random ordering of this random set of edges.
  Write $F_i=\{f_1,\ldots,f_i\}$ for $0\le i\le m$.
  We add the random edges in $F_m$ one by one (edge $f_i$ at time $t_i$),
  so the distribution of $f_i$ is uniform over $F\sm F_{i-1}$.
  Some of these edges will be coloured {\em blue} according to a rule explained below.
  For $i\in[m]$,
  let $K_i=Y_k^{t_i-1}\cap L$ be the set of vertices of degree at least $k$ in $L$ just before time $t_i$.
  We note by \cref{lem:gnt:large}, \cref{obs:ok:deg} and the discussion above that
  $|K_i| \le 2t_i/k \le 2(a+b/k)n$ for every $i\in[m']$, \whp{}.
  Note that the graph that consists of the red edges at time $\tau_2$ is $\cO_{\ell,2}$.
  For every $0\le i\le m$,
  denote by $H_i$ the graph that consists of the red edges (at time $\tau_2$)
  and the blue edges among $F_i$.
  In particular, $H_0=\cO_{\ell,2}$, and for every $i\ge 0$, $H_i$ is a supergraph of $\cO_{\ell,2}$.
  By \cref{cor:ok:exp,cor:ok:conn,lem:boosters},
  and since expansion (and connectivity) is monotone,
  we know that \whp{}, for all $i$,
  either $H_i$ is Hamiltonian,
  or there are at least $n^2/32$ boosters with respect to it.
  Suppose $H_i$ is not Hamiltonian, and let $Z_i$ be the set of boosters with respect to $H_i$.
  Let $Z'_i=Z_i\cap \binom{L}{2}\sm \binom{K_i}{2}\sm F_i$.
  These are the boosters that are
  (a) contained in $L$,
  (b) have at least one endpoint of degree less than $k$ at the time of arrival,
  and (c) have not been revealed earlier.
  It follows that
  \[
    |Z'_i|\ge |Z_i| - |E(X_\ell^{t_0},V)| - \left|\binom{K_i}{2}\right| - |F_i|
    \ge \left(\frac{1}{32}-\frac{1}{100}-2\left(a+\frac{b}{k}\right)^2 - o(1)\right)n^2.
  \]
  This is at least $n^2/100$
  for $\ell\ge 500$ and $k=100\ell$.
  When considering the edge $f_{i+1}$ we colour it blue
  if it is in $Z'_i$, or if it is in $\binom{L}{2}\sm \binom{K_i}{2}$
  and $H_i$ is already Hamiltonian.
  We observe that since $H_0=\cO_{\ell,2}\subseteq\cO_{k,2}$,
  and since every blue edge is in $\cO_{k,2}$,
  we have that $\cH_i\subseteq \cO_{k,2}$.
  Note that the probability that it is in $Z_i'$ is at least
  $|Z_i'|/(|F|-|F_i|)\ge q:=1/100$,
  independently of the past.
  Thus,
  using \cref{prop:chernoff},
  the number of blue edges stochastically dominates a binomial random variable with
  $m'\sim bnp/2$ attempts and success probability $q$,
  and hence mean $\sim qbnp/2>2n$.
  By Chernoff bounds (\cref{thm:chernoff}), the number of blue edges is at least $n$ \whp{}.
  We finish the proof by noting that if $H_i$ is not Hamiltonian and $f_{i+1}$ is coloured blue
  then $f_{i+1}$ is a booster with respect to $H_i$ that is contained in $\cO_{k,2}$.
  Thus, $H_{m'}$ (and thus $H_m$; and thus $\cO_{k,2}$) is \whp{} Hamiltonian.
\end{proof}

The proof above yields the following corollary.
\begin{corollary}\label{cor:ok:ham}
  Let $k_H$ be the constant from \cref{thm:ok:ham:hit}.
  Then, for every $k\ge k_H$,
  $\cO_{k,2}$ is \whp{} a Hamiltonian $\frac{n}{4}$-expander.
\end{corollary}

\section{Spanning subgraphs}\label{sec:span}
\subsection{Minimum degree and vertex-connectivity}\label{sec:mindeg}
It follows from \cref{cor:ok:size} that Builder can construct a graph with minimum degree $k$ by purchasing
(sufficiently) more than $\co_k n$ edges,
while observing only enough edges to guarantee a sufficient minimum degree in the underlying random graph process.

\begin{proof}[Proof of \cref{thm:mindeg:hit}]
  We describe Builder's strategy.
  Builder purchases any edge touching at least one vertex of degree less than $k$.
  Since by time $\tau_k$ the random graph process has, by definition, minimum degree of at least $k$,
  the obtained graph is exactly $\cO_k$.
  We conclude by observing that by \cref{cor:ok:size} Builder purchased, \whp{},
  at most $\co_kn+O(\sqrt{n}\log{n})$ edges,
  and that for $k\ge 3$, by \cref{thm:ok:conn}, $\cO_k$ is $k$-vertex-connected.
\end{proof}

\begin{proof}[Proof of \cref{thm:mindeg}]
Let $\eps'=\eps/2$.
We describe Builder's strategy in stages.
\paragraph{Stage I (Constructing a $k$-matching)}
\stage{$Cn$}{$(k-\eps')n/2$}
During the first $Cn$ steps, for $C=C(\eps',k)$, Builder constructs a $k$-matching,
in which all but at most $\eps' n$ vertices are of degree $k$.
This is possible, \whp{}, by \cref{lem:kmatching}.

\paragraph{Stage II (Handling low-degree vertices)}
\stage{$(1+\eps')n\log{n}/2$}{$\eps'kn$}
Denote by $V_0$ the set of vertices of degree smaller than $k$ in Builder's graph at the end of Stage~I.
By the above, \whp{}, $|V_0|\le \eps' n$.
Builder now emulates the $\cO_k$ model with respect to these vertices;
namely, Builder purchases any edge contained in $V_0$ that is incident to at least one vertex of degree less than $k$ in $G[V_0]$.
Since in the random graph process, by time $t$, the minimum degree is logarithmic in $n$,
Builder will observe at least $k$ additional edges incident to each vertex of $V_0$.
He will purchase a subset of these (thus, at most $\eps' kn$ edges).
\end{proof}

\subsection{Hamilton cycles}\label{sec:ham}
\subsubsection{Hitting time, inflated budget}\label{sec:ham:hit}
\begin{proof}[Proof of \cref{thm:ham:hit}]
  Let $k$ be a large enough constant so that $\cO_{k,2}$ is
  \whp{} Hamiltonian
  (such $k$ is guaranteed to exist by \cref{thm:ok:ham:hit}).
  Builder emulates $\cO_{k,2}$ by purchasing every edge that is incident to a vertex of degree less than $k$.
  He completes this in time $\tau_2$ while purchasing at most $kn$ edges.
\end{proof}

\subsubsection{Inflated time, optimal budget}\label{sec:ham:inflated}
In this section, we prove the following theorem:
\begin{theorem}[Hamiltonicity]\label{thm:ham}
  For every $\eps>0$ there exists $C>1$ such that the following holds.
  If $t\ge Cn\log{n}$ and $b\ge(1+\eps)n$ then there exists a $(t,b)$-strategy $B$ of Builder such that
  \[
    \lim_{n\to\infty} \pr(B_t\text{ is Hamiltonian}) = 1.
  \]
\end{theorem}

\paragraph*{Proof of \cref{thm:ham}}
We describe Builder's strategy in (four) stages.
Set $\eps'=\eps/40$.
We describe a $(t,b)$-strategy for $b=(1+\eps)n$ and $t<n\log{n}/\eps'$.

\paragraph{Stage I (Growing disjoint paths)}
\stage{$\Theta(n\log^{1/3}{n})$}{$(1-\eps')n$}
Let $C$ be a large constant to be chosen later.
In the first stage, Builder grows many sublinear paths, covering together $(1-\eps')n$ vertices.
Let $s'\sim n/\log^{1/3}{n}$ be an integer
and let $s_0=n/\log^{1/2}{n}\ll s'$.
Builder grows simultaneously paths $P_1,\ldots,P_{s'}$ as follows.
He begins by letting each $P_i$ be a (distinct) vertex $v^i_0$.
He then claims every edge that extends one of the paths without intersecting with other paths.
Formally, let $\nu=(1-\eps')n/s'$.
If at a given stage Builder has the paths
$(v^1_0,\ldots,v^1_{\ell_1}),\ldots,(v^{s'}_0,\ldots,v^{s'}_{\ell_{s'}})$ on the vertex set $V_P$
then he claims an observed edge if and only if
it is of the form $\{v^j_{\ell_j},w\}$
for some $j=1,\ldots,s'$, $\ell_j<\nu$
and $w\notin V_P$.
Builder stops if all but at most $s_0$ of the paths are of length at least $\nu$
(in which case the stage is ``successful''),
or when he has observed $t_1=\frac{C}{\eps'}n\log^{1/3}{n}$ edges (in which case the stage ``fails''),
whichever comes first.
For $1\le j\le s'$ and $1\le i\le t_1$,
let $y^j_i$ be the indicator that Builder has purchased the $i$'th observed edge,
and that this edge extends the path $P_j$.
For convenience, if $|V_P|\ge (1-\eps')n$ we set $y^j_i$ to $1$.
Observe that for every $j,i$,
$\pr(y^j_i=1) \ge (n-|V_P|)/\binom{n}{2}>\eps'/n$.
If the stage fails then there is an $s_0$-subset $S_0\subseteq[s']$
such that for every $j\in S_0$, the path $P_j$ ends up with length less than $\nu$.
Write $y_0 = \sum_{j\in s_0} \sum_{i=1}^{t_1} y^j_i$.
Thus, if the stage fails then $y_0<\nu s_0$. 
But $y_0$ stochastically dominates a binomial random variable with $t_1$ attempts
and success probability $p_1=\eps's_0/n$.
Set $\mu=t_1p_1\sim Cs_0\log^{1/3}{n}\gg s'$.
By the union bound over the choices of an $s_0$-subset of the paths
and using Chernoff bounds (\cref{thm:chernoff}), 
for large enough $C$,
the probability that the stage fails is at most
\[
  \binom{s'}{s_0} \pr(y_0 < \nu s_0)
  \le 2^{s'} \exp(-\mu/10) = o(1).
\]

\paragraph{Stage II (Connecting the paths)}
\stage{$n\log^{1/2}{n}$}{$o(n)$}
At this stage Builder tries to connect most of the $s=s'-s_0\sim s'$ complete paths to each other,
eventually obtaining an almost-spanning path in his graph.
For convenience, we re-enumerate the complete paths by $[s]$.
Let $t_2=n\log^{1/2}{n}$. 
During the next $t_2$ observed edges,
Builder purchases an edge connecting the last vertex of one of his paths to the first vertex of another.
Formally, he purchases every observed edge of the form $\{v^i_{\ell_i},v^j_0\}$ for $i\ne j$.
We will need the following lemma:
\begin{lemma}[\cite{BKS12}*{Lemma 4.4}]\label{lem:DFS}
  Let $s,k\ge 1$ and let $D$ be an $s$-vertex digraph
  in which for every two disjoint $A,B\subseteq V(D)$ with $|A|,|B|\ge k$,
  $D$ contains an edge between $A$ and $B$.
  Then, $D$ contains a directed path of length $s-2k+1$.
\end{lemma}

We use the lemma on an auxiliary digraph $D$ on the vertex set $[s]$ that is defined as follows.
We observe $t_2$ random edges.
Whenever Builder observes (and purchases) an edge of the form $\{v^i_{\ell_i},v^j_0\}$,
we add the edge $(i,j)$ to $D$.
We now have to claim that $D$ satisfies the requirements of the lemma for $k=\eps' s$.

\begin{claim}\label{cl:D:exp}
  \Whp{},
  every two disjoint $A,B\subseteq[s]$ with $|A|,|B|\ge \eps' s$ satisfy
  ${E_D(A,B)\ne\es}$.
\end{claim}

\begin{proof}
  Fix disjoint $A,B$ with $|A|,|B|\ge \eps' s$.
  The event that $E_D(A,B)=\es$ implies that none of the $t_2$ observed edges hits a pair $\{v^i_{\ell_i},v^j_0\}$ for $i\in A$ and $j\in B$.
  There are at least $\eps'^2 s^2$ such edges, hence that probability is at most
  \[
    \binom{\binom{n}{2}-\eps'^2 s^2}{t_2}/\binom{\binom{n}{2}}{t_2}
    \le \exp\left(-\eps'^2 s^2 t_2/\binom{n}{2}\right)
    = \exp(-\Theta(n\log^{-1/6}{n})).
  \]
  Here we used the general bound $\binom{a-b}{c}/\binom{a}{c}\le\exp(-bc/a)$.
  By the union bound over the choices of $A,B$, the probability that the event in the statement does not hold is at most
  \[
    \binom{n}{\eps' s}^2 \exp(-\Theta(n\log^{-1/6}{n}))
    \le \exp(\eps' s\log\log{n} -\Theta(n\log^{-1/6}{n})) = o(1).\qedhere
  \]
\end{proof}

\Cref{cl:D:exp,lem:DFS} imply that $D$ has, \whp{}, a directed path of length $(1-2\eps')s$.
By the way we have defined $D$, this implies the existence of a path $Q$ of length $(1-3\eps')n$ in Builder's graph.
To analyse the number of purchased edges at this stage,
we note that at any stage an edge is purchased with probability $\asymp s^2/n^2$,
hence the expected number of purchased edges is $\asymp t_2s^2/n^2\asymp n/\log^{1/6}{n}$.
By Chernoff bounds (\cref{thm:chernoff}), the number of purchased edges is sublinear \whp{}.

\paragraph{Stage III (Preparing the ground)}
\stage{$\frac{1}{\eps'}n\log{n}$}{$36\eps' n$}
Let $q_1,q_2$ be the endpoints of $Q$.
Let $V_1\cup V_2=V\sm V(Q)$ be a partition of the vertices outside $Q$
(so $|V_i|=3\eps' n/2$ for $i=1,2$).
Builder performs the next two tasks simultaneously.

\subparagraph{(Connecting the endpoints of $Q$ to $V_1,V_2$)}
Builder claims the first observed edge from $q_i$ to $V_i$ for each $i=1,2$
(unless such an edge already exists).
After $cn\log{n}$ steps, for any $c>0$, there will be \whp{} a neighbour of each of $q_i$ in $V_i$.
Call this neighbour $w_i$.

\subparagraph{(Constructing expanders on $V_1,V_2$)}
Builder purchases a copy of $\cO_{12}$ in $V_i$.
Observing that every new edge falls inside $V_i$ with probability at least $2\eps'^2$,
we conclude, using \cref{thm:chernoff,lem:gnt:deg:hit,cor:ok:size},
that this could be done \whp{} by observing at most $\frac{1}{\eps'}n\log{n}$ edges
and purchasing at most $36\eps' n$ of them.
By \cref{lem:ok:exp}, the obtained graphs are both $\beta n$-expanders for some $\beta>0$.
For $i=1,2$ let $B_i=B[V_i]$ denote Builder's graph on the vertex set $V_i$.

\paragraph{Stage IV (Sprinkling)}
\stage{$O(n)$}{$3\eps'n$}
By \cref{lem:sprinkling,thm:chernoff} there exists $C_4=C_4(\beta,\eps')>0$
such that by observing at most $C_4n$ edges
and purchasing at most $\frac{3}{2}\eps'n$ edges among those landing inside each $V_i$,
both $B_1,B_2$ become \whp{} Hamiltonian.
Conditioning on that event,
denote the sets of endpoints of Hamilton cycles in $B_i$ whose other endpoint is $w_i$ by $Y_i$
for $i=1,2$.
By \cref{lem:manyends},
$|Y_i|\ge \beta n$.

\paragraph{Stage V (Closing a Hamilton cycle)}
\stage{$\log{n}$}{$1$}
To close a Hamilton cycle, Builder purchases an edge between $Y_1$ and $Y_2$.
The probability of failing to do so after observing $\log{n}$ steps is $o(1)$.
\qed

\subsection{Perfect matchings}\label{sec:pm}
In this section we prove \cref{thm:pm}.
Recall the constant $k_H$ from \cref{thm:ok:ham},
and let $\eps'=\eps/k_H$.
To aid in understanding the proof, we refer the reader to
\cref{fig:pmprf},
which illustrates the inclusion relationships and relative sizes
of sets constructed during the first three stages.

\paragraph{Stage I (Constructing a $k$-core)}
\stage{$\Theta(n)$}{$\eps'n$}
Let $k$ be some large (even) constant to be determined later (in Stage~VI).
Choose $\eps_1=\eps'/(10k)$
and let $V_1$ be an arbitrary vertex set of size $\sim \eps_1 n$.
During the first $t_1\sim 9\eps_1^{-1}kn$ steps Builder purchases every edge that lands inside $V_1$.
By that time, \whp{}, at least $8k\eps_1 n$
(and at most $10k\eps_1 n=\eps'n$) edges
are being purchased.
Thus, by \cref{cor:core}, \whp{}, the obtained graph contains a $k$-core
on vertex set $U_1\subseteq V_1$
of size at least $\eps_1n/2$.
Write $\eta_1=|U_1|/n$
and note that $\eta_1\le\eps_1\le\eps'/10$.

\paragraph{Stage II (Constructing a large matching)}
\stage{$\Theta(n)$}{$n/2$}
Let $\eps_2=\eps'/3$.
Let $Y=V\sm U_1$;
so $|Y|=(1-\eta_1)n$. 
Let $C_2=C_2(1,\eps_2)$ be the constant guaranteed by \cref{lem:kmatching}.
Builder follows a $(t_2,b_2)$-strategy proposed by \cref{lem:kmatching}
to construct a matching $M_2$ with vertex set $X_2\subseteq Y$
where $|X_2|\ge|Y|(1-\eps_2)\sim(1-\eps_2)(1-\eta_1)n$.
The lemma guarantees that this is doable, \whp{}, for $t_2\sim C_2n$
and $(1-\eps)n/2 \le b_2\le n/2$.
Write $\eta_2=1-|X_2|/n$,
and note that $\eta_1 \le \eta_2 \le \eta_1+\eps_2$.

\paragraph{Stage III (Extending the large matching)}
\stage{$\eps'n\log{n}$}{$\eps n/3$}
Denote the vertices outside $X_2\cup U_1$ by $W_2$.
Recall that $|W_2|= n-(|X_2|+|U_1|)=(\eta_2-\eta_1)n$.
Write $\eta_3=\eta_2-\eta_1\le\eps_2$.
During the next $t_3\sim \eps'n\log{n}$ steps
Builder (partially) constructs $\cO_{k_H}$ on $W_2$.
The number of purchased edges is \whp{} at most
$k_H\eta_3n\le k_H\eps_2 n\le
\eps n/3$.
Let $S$ be the set of vertices of degree less than $k_H$ in $W_2$
at the end of this process.
We now show that, \whp{}, $|S|\le n^{1-\delta_3}$ for $\delta_3=\eta_3\eps'/2$.
Indeed, let $w\in W_2$.
The probability that the next edge contains $w$ and is contained in $W_2$
is $p_3\sim 2\eta_3/n$.
Thus, the total number of observed edges incident to $w$
stochastically dominates a binomial random variable with $t_3$ attempts
and success probability $p_3$.
Consequently, the probability that edges containing $w$ were observed less than $k_H$ times
is, by \cref{cl:bin:lowtail}, at most $n^{-\eta_3\eps'}$.
Hence, by Markov's inequality, $\pr(|S|\ge n^{1-\delta_3}) \le \eta_3 n^{1-\eta_3\eps'}/n^{1-\delta_3}\ll 1$.
We argue that $B[W_2]$ contains a matching
that covers all but $2k_Hn^{1-\delta_3}$ vertices.
Indeed,
if we had continued constructing the copy of $\cO_{k_H}$ without any time restrictions,
we would have had -- according to \cref{thm:ok:ham} -- a (nearly) perfect matching there.
But, deterministically, we would have purchased at most $k_H|S|$ additional edges to achieve that goal.
Thus, \whp{}, the current largest matching covers all but at most $2k_H n^{1-\delta_3}$ vertices in $W_2$.
Denote this matching by $M_3$, and let $X_3=V(M_3)$.

\begin{figure*}[t!]
  \captionsetup{width=0.879\textwidth,font=small}
  \centering
  \begin{tikzpicture}
    \clip (-0.2,-0.2) rectangle (10.2,5.6);

    \draw[rounded corners] (0,0) rectangle (10,5); 
    \node (V) at (0.25,5.25) {$V$};
    \draw[dashed,rounded corners] (0.1,0.1) rectangle (3,4.9); 
    \node[rotate=90] (V1) at (2.75,2.5) {$V_1$};
    \draw[rounded corners] (0.2,0.2) rectangle (1.5,4.8); 
    \node[rotate=270] (U1) at (0.45,2.5) {$|U_1|=\eta_1n$};
    \draw[dotted,rounded corners] (1.6,0.2) rectangle (9.9,4.8); 
    \node (Y) at (5.75,4.45) {$|Y|=(1-\eta_1)n$};
    \draw[blue,densely dashdotted,rounded corners] (1.7,1.5) rectangle (9.8,4.7); 
    \node[blue] (X2) at (5.75,1.75) {$|X_2|=(1-\eta_2)n$};
    \draw[rounded corners] (1.7,0.3) rectangle (9.8,1.4); 
    \node (W2) at (5.75,1.05) {$|W_2|=(\eta_2-\eta_1)n$};
    \draw[blue,densely dashdotted,rounded corners] (2.2,0.4) rectangle (9.7,1.3); 
    \node[blue] (X3) at (8.75,0.65) {$|X_3|\sim|W_2|$};

    \foreach \x in {0,1,3,4} {
      \draw[red] (2.6+1.5*\x,1.75) -- ($(2.6+1.5*\x,1.75)+(30:1)$);
    }
    \foreach \x in {0,1,3,4} {
      \draw[red] (2+1.5*\x,2.625) -- ($(2+1.5*\x,2.625)+(30:1)$);
    }
    \foreach \x in {0,...,4} {
      \draw[red] (2.3+1.5*\x,3.5) -- ($(2.3+1.5*\x,3.5)+(30:1)$);
    }
    \node[red] (M2) at (5.75,3.1) {$M_2$};
    \foreach \x in {0,...,3} {
      \draw[red] (2.3+1.5*\x,0.5) -- ($(2.3+1.5*\x,0.5)+(15:1)$);
    }
    \node[red] (M3) at (3.4,1) {$M_3$};
  \end{tikzpicture}
  \caption{A diagram illustrating
           the inclusion relationships and relative sizes
           of the sets $V,V_1,U_1,Y,X_2,W_2,X_3$
           constructed during Stages I--III in the proof of \cref{thm:pm},
           in addition to the location of the matchings $M_2,M_3$.
           }
  \label{fig:pmprf}
\end{figure*}

\paragraph{Stage IV (Building stars)}
\stage{$(1+\eps')n\log{n}/2$}{$o(n)$}
We append the matching constructed in Stage~III to the matching constructed in Stage~II;
namely, we set $M=M_2\cup M_3$, $X=X_2\cup X_3$ and $W=V\sm(X\cup U_1)$.
We recall that $|W|\le n^{1-\delta_4}$ for some $0<\delta_4<\delta_3$
and $|X|\sim(1-\eta_1)n$.
Let $K>1/(2\eta_1\eps')$ be an integer.
In this stage Builder attempts to construct disjoint $K$-leaf stars at each vertex of $W$,
with leaves inside $X$, where each edge of $M$ contains at most one leaf.
For $w\in W$ let $x_w$ denote the number of leaves in the star that is rooted at $w$
that Builder managed to purchase by the end of the stage.
We show that \whp{} for every $w\in W$, $x_w=K$, hence the stage is successful.
Indeed, whenever an edge arrives it has probability $2/n$ to be incident to $w$.
The other end of such an edge is incident to an available edge of $M$
(i.e., an edge which does not already contain a leaf of a star constructed at this stage)
with probability at least $(|X|-2K|W|)/n\sim 1-\eta_1$.
Thus, $x_w$ stochastically dominates a binomial random variable with
$t_4=(1+\eps')n\log{n}/2$ attempts
and success probability $p_4=(1-2\eta_1)\cdot 2/n$.
Therefore, by \cref{cl:bin:lowtail} (since $\eps'<1$ and $2\eta_1<\eps'$),
\[
  \pr(x_w<K)
  \le \left(\frac{et_4p_4}{K(1-p_4)}\right)^K e^{-t_4p_4}
  = (\Theta(\log{n}))^K \cdot n^{-(1+\eps')(1-2\eta_1)}
  \ll n^{-1}.
\]
By the union bound over all $o(n)$ vertices of $W$,
the probability that any $x_w$ is below $K$ is $o(1)$.

\paragraph{Stage V (Building $3$-paths)}
\stage{$\eps' n\log{n}$}{$o(n)$}
For every $w\in W$ let $a^w_1,\ldots,a^w_K\in X$ be the leaves of the stars rooted at $w$
that were constructed in the previous stage.
For $i=1,\ldots,K$, let $b^w_i\in X$ be the neighbour of $a^w_i$ in $M$.
Builder now purchases every edge that connects $b^w_i$
(for some $w\in W$ and $i\in[K]$)
to a vertex $u\in U_1$,
as long as none of $b_1^w,\ldots,b_K^w,u$ was touched earlier during this stage.
Fix $w\in W$.
The probability that in a given step an edge between (an untouched vertex of)
$U_1$ and $\{b_1^w,\ldots,b_K^w\}$ appears
is at least $p_5=K(|U_1|-|W|)/\binom{n}{2}\sim 2K\eta_1/n$.
Thus, the probability that in $t_5=\eps'n\log{n}$ steps such an edge does not appear
is at most $e^{-p_5t_5}=n^{-2K\eta_1\eps'}$.
Noting that $K>(2\eta_1\eps')^{-1}$ we get that this probability is $o(n^{-1})$,
hence, by the union bound over all $o(n)$ vertices of $W$,
we match one leaf in the star rooted at $w$ with a unique vertex in $U_1$ for every $w\in W$ \whp{}.

\paragraph{Stage VI (Sprinkling)}
\stage{$\Theta(n)$}{$\eps'n$}
We note that each of the $3$-paths constructed in the previous stage can be used as an augmenting path.
Namely, we let $M'$ be the matching obtained by $M$ by replacing every edge $e\in M$
that is contained in such a $3$-path (as a middle edge) by the other two edges of that path.
Let $U'$ be the set of vertices in $U_1$ matched in $M'$,
and write $U^*=U_1\sm U'$.
Thus, $M'$ matches all vertices but those in $U^*$.
We now show that $B[U^*]$ is \whp{} an expander.
For that, note that $U'$ is a uniformly chosen subset of size $|W|\le n^{1-\delta_4}$ of $U_1$.
We first argue that $\delta(B[U^*])\ge k/2$ \whp{}.
Indeed, fix $u\in U_1$, and let $u_1,\dots,u_k$ be arbitrary neighbours of $u$ in $B[U_1]$.
The probability that $d(u,U^*)\le k/2$ is at most
the probability that $k/2$ of these $k$ neighbours end up in $U'$.
By the union bound, this is at most $\binom{k}{k/2}n^{-\delta_4}$
\[
  \binom{k}{k/2} \frac{\binom{|U_1|-k/2}{|W|-k/2}}{\binom{|U_1|}{|W|}}
  \le \left(4\cdot\frac{|W|}{|U_1|}\right)^{k/2}
  \asymp n^{-\delta_4k/2}.
\]
Making sure $k>2/\delta_4$, this is $o(n^{-1})$,
and hence, by the union bound, \whp{}, $\delta(B[U^*])\ge k/2$.
We remark that there is no circular dependency of the constants here.

We recall that by \cref{lem:span}
(taking $t=t_1\sim 9kn$ and $R=3\beta_6 n$ for $\beta_6=3^{-5}k^{-3/2}$)
the random graph process $G_t$, and hence also $B[U^*]$,
typically does not have a set of size $r\le 3\beta_6 n$ that spans more than $3r$ edges.
This, together with the minimum degree of $B[U^*]$, implies expansion;
indeed, let $A\subseteq U^*$ be a vertex set with $|A|=a\le R/3 = \beta_6n$.
Suppose $|N(A)|<2|A|$ and write $S=A\cup N(A)$ (so $r=|S|<3a\le R$).
But the number of edges spanned by $S$ is at least $\delta(B[U^*])|A|/2\ge ka/4 > kr/12\ge 3r$
(making sure $k\ge 36$),
a contradiction to \cref{lem:span}.
Thus, $B[U^*]$ is a $\beta_6n$-expander.
By \cref{lem:sprinkling} there exists $C_6=C_6(\beta_6)$
such that adding to $B[U^*]$ $C_6\eta_1 n$ random non-edges
makes it, \whp{}, Hamiltonian.
Letting $t_6=C_6'n/\eta_1$ for $C_6'>C_6$,
it follows from \cref{thm:chernoff} that after observing $t_6$ edges,
at least $C_6\eta_1 n$ of them land, \whp{}, inside $U^*$ (whose size is $\sim\eta_1 n$).
Builder purchases each of these edges.
Thus, by \cref{lem:sprinkling},
Builder's resulting graph on $U^*$ contains, \whp{}, a Hamilton cycle.
That Hamilton cycle contains a perfect matching $M_6$.
Appending $M_6$ to $M$, we get a perfect matching in Builder's graph.
\qed

\section{Small subgraphs}\label{sec:small}

\subsection{Trees}\label{sec:trees}

\subsubsection*{$1$-statement}
\begin{proof}[Proof of the $1$-statement of \cref{thm:trees}]

  Let $T_1\subseteq\ldots\subseteq T_k$ be a sequence of subtrees of $T$,
  $T_i$ having $i$ vertices (so $T_k=T$).
  We first note that if $t\gg n$ then a budget $b=k-1$ suffices with high probability.
  Indeed, 
  suppose Builder has constructed $T_i$ (note that he builds $T_1$ with a zero budget).
  To build $T_{i+1}$, Builder waits for one of at least $\sim n$ edges
  whose addition to his current copy of $T_i$ would create a copy of $T_{i+1}$.
  By time $t/(k-1)$ Builder observes (and purchases) such an edge \whp{}.
  Thus, by time $t$ Builder constructs, \whp{}, a copy of $T$.

  We may thus assume from now on that $t\le cn$
  for some $0<c<1/5$ and $t\ge b\gg (n/t)^{k-2}$;
  by monotonicity, the result extends to the remaining range.
  In particular, $t\gg n^{1-1/(k-1)}$.
  For $i=1,\ldots,k$ define
  \[
    s_i = \frac{b}{k-1}\cdot\left(\frac{t}{(k-1)n}\right)^{i-2},
  \]
  and note that $s_i$ is a strictly monotone decreasing sequence
  with $n\ge s_1 > s_k\gg 1$.
  We describe Builder's strategy in stages
  (for convenience, the first stage will be indexed by $2$).
  We assume that at the beginning of stage $i$, $i=2,\ldots,k$,
  Builder has built $s_{i-1}$ vertex-disjoint copies of $T_{i-1}$.
  The assumption trivially holds for $i=2$, since $s_1\le n$.
  At stage $i\ge 2$, Builder purchases any edge extending one of the $s_{i-1}$ copies of $T_{i-1}$
  he currently has to a copy of $T_i$,
  so that the resulting trees are still vertex-disjoint.
  He continues doing so only as long as he does not have $s_i$ copies of $T_i$
  (in which case this stage is successful)
  or if he had observed more than $t/(k-1)$ edges during this stage
  (in which case this stage fails, and thus the entire strategy fails).
  For stage $i$ and $1\le j\le t/(k-1)$ let $y^i_j$ denote be the event
  that the $j$'th observed edge of stage $i$ is being purchased,
  or that $s_i$ copies of $T_i$ have already been built.
  Being at stage $i$, there are at least $s_{i-1}-s_i$ potential trees to extend,
  each can be extended by observing one of at least $n-2b$ edges.
  Thus,
  noting that $s_{i-1}-s_i\ge s_{i-1}(1-c/(k-1))$
  and $n-2b\ge n-2t \ge n(1-2c)$,
  and conditioning on all possible values of $y^i_{j'}$ for $1\le j'<j$,
  we have
  \[
    \pr(y^i_j) \ge
    \frac{(s_{i-1}-s_i)(n-2b)}{\binom{n}{2}}
    \ge \frac{2(1-\frac{c}{k-1})s_{i-1}\cdot(1-2c)}{n}.
  \]
  Write $c'=2(1-c/(k-1))(1-2c)$ and note that for small enough $c>0$ ($c<1/5$ suffices) we have $c'>1$.
  Therefore (by \cref{prop:chernoff}), $y^i=\sum_j y^i_j$ stochastically dominates a binomial random variable
  with mean $\frac{t}{k-1}\cdot\frac{c's_{i-1}}{n} = c's_i$.
  By Chernoff bounds (\cref{thm:chernoff}), the stage is successful \whp{}.
  Thus, Builder builds, \whp{}, $s_k\ge 1$ copies of $T_k=T$ by the end of $k-1$ stages, that is, by time $t$.
\end{proof}

\subsubsection*{$0$-statement}
\begin{proof}[Proof of the $0$-statement of \cref{thm:trees}]
  We begin by assuming that $b\ge k-1\ge 2$;
  otherwise, the statement is trivial.
  We also assume that $b\ll (n/t)^{k-2}$,
  which implies that $t\ll n$.
  Moreover, we may assume that $b\ll t$.
  Indeed, suppose that $b\asymp t$.
  Then, from $b\ll (n/t)^{k-2}$,
  it follows that $t\ll (n/t)^{k-2}$,
  leading to $t\ll n^{1-1/(k-1)}$.
  This implies that \whp{} no tree on $k$ vertices exists in the graph of the observed edges.

We prove that any $(t,b)$-strategy of Builder fails, \whp{}, to build a connected component of size $k$.
Define a sequence $(r_i)_{i=1}^{k-1}$ by $r_1=n$ and $r_i=b(3kt/n)^{i-2}$ for $i=2,\ldots,k-1$.
We argue that \whp{} Builder cannot build more than $r_i$ connected components of size $i$.
The proof is by induction on $i$.
The case $i=1$ is trivial.
The case $i=2$ follows deterministically by the budget restriction, as $r_2=b$.
To build a connected component of size $2<i<k$,
Builder must observe an edge connecting a connected component of size $1\le j<i$ with a connected component of size $i-j\ge j$.
By induction, the probability that the next observed edge is such,
conditioning on any previous sequence of failures and successes, is at most
\[
  \frac{\sum_{j=1}^{\floor{i/2}} (jr_j \cdot (i-j)r_{i-j})}{\binom{n}{2}-t}
  \le \frac{inr_{i-1} + \Theta\left(b^2(t/n)^{i-4}\right)}{\binom{n}{2}-t}
  \le \frac{2.5kr_{i-1}}{n}.
\]
Here we used the fact that $b\ll t$, hence $b^2(t/n)^{i-4}\ll nr_{i-1}$.
Thus, by Chernoff bounds (\cref{thm:chernoff} and \cref{prop:chernoff}),
during $t$ rounds the number of such edges is, \whp{}, smaller than $r_{i-1}\cdot 3kt/n=r_i$.
To build a tree of order $k$, Builder must construct a connected component of size at least $k$.
For that, Builder must observe an edge connecting a connected component of size $1\le i\le k-1$
with a connected component of size $\max\{i,k-i\}\le j\le k-1$ (in his graph).
The probability that the next observed edge is such is at most
\[
  \frac{\sum_{i=1}^{k-1}\sum_{j=\max\{i,k-i\}}^{k-1} (ir_i\cdot jr_j)}{\binom{n}{2}-t}
  \preceq \frac{r_{k-1}}{n}.
\]
Hence, during $t$ rounds the expected number of edges Builder observes that create such a connected component
is $\preceq r_{k-1}t/n\ll 1$;
thus, Builder fails to construct such a component \whp{}.
\end{proof}

\subsection{Cycles}\label{sec:cycles}
For a fixed graph $H$ we call a non-edge $e$ in Builder's graph an \defn{$H$-trap}
(or, simply, a trap)
if $B+e$ contains a copy of $H$.

\subsubsection*{$1$-statement}

\begin{proof}[Proof of the $1$-statement of \cref{thm:cycles}]

  Let $b\gg b^*=b^*(n,t,k)$.
  As $b^*\ll n$, we may assume $b\ll n$.
  Choose $n\ll t'\ll t$ such that it still holds that $b\gg b':=b^*(n,t',k)$.
  Set $s=n^2/(b't')=\min\{(t'/n)^k,n/\sqrt{t'}\}$,
  and note that $1\ll s\ll b$.
  Indeed, $b/s\gg b'/s = n^2/(s^2t')\ge 1$.
  Denote $d=cs^{1/k}$ for $c<1/(k+5)$
  and note that $d\gg 1$.
  Note further that $dn=cs^{1/k}n\le ct'$
  and that $c^k<1/6$.

  \paragraph{Stage I (Growing $d$-ary trees)}
  \stage{$t'$}{$b/2$}
  Builder performs Stage I in $k$ rounds, each lasting $dn$ steps.
  Set $r=b/s\gg 1$.
  Builder chooses $2^k r$ vertices arbitrarily, call them {\em roots}.
  Denote by $P_0=\{v_1,\dots,v_{2^kr}\}$ the set of these roots,
  and let $L_0=R_0=P_0$
  (we think of each root as being both on the right side of a tree and on the left side of a tree).
  Suppose at the beginning of round $i\in[k]$
  we have a set $J_{i-1}\subseteq[P_0]$ with $|J_{i-1}|=2^{k-(i-1)}r$
  and trees $(T_j)_{j\in J_{i-1}}$,
  where the tree $T_j$ is rooted at $v_j$ and is of depth $i-1$
  (so at the beginning of round $1$ these trees are isolated roots).
  Write $L_i^j$ for the set of vertices in $L_i$ which are in tree $j$,
  and define $R_i^j$ analogously.
  Builder's goal at round $i$
  is to ``extend'' {\em half} of these trees,
  by attaching $d|L_{i-1}^j|$ leaves to vertices in $L_{i-1}^j$
  and $d|R_{i-1}^j|$ (distinct) leaves to vertices in $R_{i-1}^j$.
  Each of the new leaves should not be already covered by Builder's graph.
  A leaf attached to $L_{i-1}^j$ will be placed in $L_i^j$,
  and a leaf attached to $R_{i-1}^j$ will be placed in $R_i^j$.
  Note that if $i=1$ then $L^j_{i-1}=R^j_{i-1}=\{v_j\}$,
  in which case we attach {\em two} $d$-leaf stars at each root $v_j$.
  Let $J_i\subseteq J_{i-1}$ be the subset of indices for which
  $T_j$ has been successfully extended.
  Set $L_i=\bigcup_{j\in J_i} L_i^j$ and $R_i=\bigcup_{j\in J_i} R_i^j$.
  Note that if Builder's strategy has not failed by the beginning of round $i$,
  $|J_{i-1}|= 2^{k-(i-1)}r$
  and $|L_{i-1}^j|=|R_{i-1}^j|=d^{i-1}$ for every $j\in J_{i-1}$.
  For $i\in[k]$ and $j\in J_{i-1}$
  let $x_i^{(j)}$ be the number of edges observed at round $i$
  with one endpoint in $L_{i-1}^j$
  and the other outside the set of vertices covered by the edges of Breaker's graph at the time of observation.
  If $|x_i^{(j)}|\ge d|L_{i-1}^j|$ then Builder purchases the first $|L_{i-1}^j|$ of those observed edges
  and this stage is ``left-successful'' for tree $j$.
  Similarly, let $y_i^{(j)}$ be the number of edges observed at round $i$
  between $R_{i-1}^j$ and the uncovered part of the graph.
  If $|y_i^{(j)}|\ge d|R_{i-1}^j|$ then Builder purchases the first $|R_{i-1}^j|$ of those observed edges
  and this stage is ``right-successful'' for tree $j$.
  If a stage is both left-successful and right-successful for a tree $j$ then $T_j$ is successfully extended.
  As we remarked earlier, after extending successfully $2^{k-i}r$ trees, Builder stops.

  Let us analyse this strategy.
  For $0\le i\le k$,
  denote by $b_i$ the number of covered vertices by the end of step $i$.
  Observe that
  \[
    b_i \le \sum_{i'=0}^i 2^{k+1-i'}rd^{i'}
    \le (1+o(1))2^{k-i+1}rd^i
  \]
  (since $d\gg 1$).
  In particular, $b_k<3rd^k=3c^kb\ll n$.
  Take $i\in[k]$ and assume the strategy was successful so far.
  Thus, $x_i^{(j)}$ (and also $y_i^{(j)}$) stochastically dominates a binomial random variable
  with $dn$ attempts and success probability at least
  $|L_{i-1}^j|\cdot(n-b_i)/\binom{n}{2} \sim 2d^{i-1}/n$.
  Thus, by Chernoff bounds (\cref{thm:chernoff}),
  $\pr(x_i^{(j)} < d|L_{i-1}^j|) = \pr(x_i^{(j)}<d^i) < \exp(-d^i/5)\ll 1$.
  By Markov's inequality, \whp{} for at least half of the trees indexed by $J_{i-1}$
  Builder succeeds in extending the tree.
  The total time that passes until the end of this stage is at most $kdn\le ckt' < t'\ll t$,
  and the total budget is at most $3c^kb < b/2$.

  \paragraph{Stage II (Connecting leaves)}
  \stage{$t-t'$}{1}
  Note that for every $j\in J_k$, if $u\in L_k^j$ and $v\in R_k^j$ then $u,v$ are both leaves of $T_j$
  of distance $2k$ from each other.
  Thus, by adding the edge $\{u,v\}$, a cycle of length $\ell$ is formed.
  Thus, the number of traps is $|J_k|\cdot|L_k^j|\cdot|R_k^j|=rd^{2k}  \asymp rs^2 = bs \gg n^2/t'$.
  Thus, after time $t-t'\gg t'$ we hit a trap \whp{}.

  \paragraph{Even cycles}
  We briefly discuss the modifications needed to obtain the result for cycles of length $2k+2$.
  Instead of choosing $P_0$, Builder greedily constructs a matching $M_0$ of size $2^kr$;
  this is doable (\whp{}) since the first $Cr$ observed edges (recalling that $r\ll n$)
  are mostly disjoint, and, in particular, for large enough $C$, contain a matching of size $2^kr$.
  Then, for each edge in $M_0$,
  Builder puts one vertex in $L_0$ and one vertex in $R_0$.
  Builder continues in the same fashion as in the case of odd cycles,
  where the only difference is that instead of $L_0=R_0$ we have $L_0\cap R_0=\es$.
  We finish by noting that an edge between $u\in L_k^j$ and $v\in R_k^j$ now closes a cycle of length $2k+2$.
\end{proof}

\subsubsection*{\texorpdfstring{$0$}{0}-statement}
Let $H=C_\ell$ for $\ell=2k+1$ or $\ell=2k+2$.
The $0$-statement in \cref{thm:cycles} follows from combining the next two claims
together with the known fact that if $t\ll n$ then $G_{n,t}$, \whp{}, contains no copy of $H$.

\begin{claim}\label{cl:cycles:lb:univ}
	If $b\ll n/\sqrt{t}$
	then for any $(t,b)$-strategy $B$ of Builder,
	\[
	\lim_{n\to\infty} \pr(H\subseteq B_t) = 0.
	\]
\end{claim}

\begin{claim}\label{cl:cycles:lb:spec}
  If $b\ll n^{k+2}/t^{k+1}$
	then for any $(t,b)$-strategy $B$ of Builder,
	\[
	\lim_{n\to\infty} \pr(H\subseteq B_t) = 0.
	\]
\end{claim}

\begin{proof}[Proof of \cref{cl:cycles:lb:univ}]
  We may assume $t\ll n^2$.
  Consider a strategy for building $H$.
  Let $B'$ be Builder's graph before obtaining the first copy of $H$.
  We show that at this point,
  the time it takes Builder to hit an $H$-trap
  (and hence construct a copy of $H$)
  is typically much larger than $t$.
  Let $X$ denote the number of $H$-traps at this point.
  Observe that every $H$-trap is a non-edge of $B'$,
  both whose endpoints are vertices of positive degree.
  Evidently, the number of vertices of positive degree is at most $2b$,
  thus $X\le 4b^2$.
  Thus, the probability of hitting a trap in $t$ steps is at most
  $tX/(\binom{n}{2}-t)\preceq tb^2/n^2\ll 1$.
\end{proof}

For the proof of \cref{cl:cycles:lb:spec} we will need a couple of lemmas.
In what follows, we may assume that $t\le n^{(k+1)/(k+1/2)}$,
as otherwise $n^{k+2}/t^{k+1}\le n/\sqrt{t}$ and the claim follows from \cref{cl:cycles:lb:univ}.

\begin{claim}\label{cl:badpaths}
  Let $\ell\ge 1$ be an integer, write $d=2t/n\gg 1$ and let $C>1$.
  Then, \whp{}, the number of paths of length $\ell$ in $G_t$ that contain a vertex whose degree is at least $Cd$
  is $o(n/d)$.
\end{claim}

\begin{proof}
  Let $L$ be the set of all vertices of degree at least $Cd$,
  and let $Q$ be the set of all paths containing a vertex from $L$.
  By Chernoff bounds (\cref{thm:chernoff}),
  \[
    \E|Q| \le n^{\ell+1} \left(\frac{t}{n^2}\right)^\ell (\ell+1) \cdot\pr(d(v)\ge Cd-2)
    \le nd^\ell e^{-cd} \ll n/d,
  \]
  for some constant $c=c(C)>0$,
  and the result follows from Markov's inequality.
\end{proof}

We turn to count the number of copies of a path of length $\ell$ in a $z$-degenerate graph $G_0$ with maximum degree $\Delta$.

\begin{lemma}\label{lem:path:embed}
  Let $z,q,\Delta,\ell\ge 1$ be integers.
  Let $G_0$ be a $z$-degenerate $q$-vertex graph with maximum degree $\Delta$.
  Then, the number of paths of length $\ell$ in $G_0$ is at most $q\cdot 2^\ell z^{\ceil{\ell/2}} \Delta^{\floor{\ell/2}}$.
\end{lemma}

\begin{proof}
  Let $v_1,\ldots,v_q$ be an ordering of the vertices of $G_0$ such that for $1\le i\le q$, $v_i$ has at most $z$ neighbours among $v_1,\ldots,v_{i-1}$.
  Given a path $P$ of length $\ell$, one of its consistent orientations has at least $k$ edges going backwards (w.r.t.\ the orientation).
  Thus, to count the number of copies of $P$ in $G_0$,
  it suffices to count the number of {\em directed} paths of length $\ell$
  with at least $\ceil{\ell/2}$ edges going backwards.
  To count those, first choose a starting vertex ($n$ options),
  then choose the locations of the edges going back (at most $2^\ell$ options),
  and finally choose from at most $z$ options for an edge pointing back,
  or at most $\Delta$ options otherwise.
  Altogether, there are at most $q\cdot 2^\ell z^{\ceil{\ell/2}} \Delta^{\floor{\ell/2}}$ such paths.
\end{proof}

\begin{proof}[Proof of \cref{cl:cycles:lb:spec}]
  Consider a strategy for building $H$.
  Let $B'$ be Builder's graph before obtaining the first copy of $H$.
  We show that at this point,
  the time it takes Builder to hit an $H$-trap
  (and hence construct a copy of $H$)
  is typically much larger than $t$.
  Let $X$ denote the number of $H$-traps at this point.
	Observe that $X$ is bounded from above by the number of paths of length $\ell-1$.
  Write $d=2t/n$ and
  let $B''$ be obtained from $B'$ by deleting every vertex of degree at least $Cd$.
  By \cref{cl:badpaths}, the number of $(\ell-1)$ paths that we removed when obtaining $B''$
  is \whp{} $o(n/d)$.
  Evidently, $\Delta(B'')\le Cd$.
	By \cref{cl:degen},
  since $b^2 n^{2k+4}/t^{2k+2}\le n^5/(40t^3)$ (here we use $t\ge 40n$),
  and since we may assume $b\gg 1$,
  $B''$ is, \whp{}, $6$-degenerate.
	Thus, by \cref{lem:path:embed}, $B'$ has, \whp{},
  at most $2b\cdot 6^{k+1}\cdot (3t/n)^k + o(n/d)\ll n^2/t$ paths of length $\ell-1$.
	Thus, \whp{}, $X\ll n^2/t$.
	Therefore, \whp{}, the probability of hitting a trap in $t$ steps is at most
  $tX/(\binom{n}{2}-t)\ll 1$.
\end{proof}

\section{Concluding remarks and open problems}\label{sec:conclude}
We have proposed a new model for a controlled random graph process
that falls into the broader category of online decision-making under uncertainty.
The question we considered is the following:
given a monotone graph property $\cP$,
is there an online algorithm that decides whether to take or leave any arriving edge
that obtains, \whp{},
the desired property within given time and budget constraints.
We analysed the model for several natural graph properties:
connectivity, minimum degree, Hamiltonicity, perfect matchings
and the containment of fixed-size trees and cycles.
Our focal point was the investigation of the inevitable trade-off
between the total number of edges Builder observes (maximum time)
and the number of edges he might need to purchase (minimum budget).
In some cases, the trade-off is substantial
(for example, to construct a triangle in a close-to-optimal time, namely,
when these just appear in the random graph process,
Builder must purchase an almost-linear number of edges).
For containment of fixed subgraphs, we have quantified that trade-off precisely.

Although we have made some sizable progress,
there are a few challenging open problems to consider.
Returning to our first result,
we have shown (\cref{thm:mindeg:hit}) that the budget needed to obtain minimum degree $k$ at the hitting time
is at most $\co_kn$.
We believe that this result is (asymptotically) optimal in the following sense:
\begin{conjecture}\label{conj:mindeg}
  For every $k\ge 2$ and $\eps>0$,
  if $b\le (\co_k-\eps) n$
  then for any $(\tau_k,b)$-strategy $B$ of Builder,
  \[
    \lim_{n\to\infty} \pr(\delta(B_{\tau_k})\ge k) = 0.
  \]
\end{conjecture}

We continued by showing (\cref{thm:mindeg}) that if we allow the time to be only asymptotically optimal,
then an asymptotically optimal budget suffices.
For large enough $k$,
a similar statement for $k$-vertex-connectivity follows from \cref{thm:mindeg:hit,cor:ok:size}.
For $k=1$, as demonstrated in the introduction, this is trivially true.
Recently,
after an earlier version of this paper appeared online,
Lichev~\cite{Lic22+} proved that this is true for every $k\ge 2$:
\begin{theorem}[\cite{Lic22+}]\label{thm:conn:lichev}
  Let $k\ge 2$ be an integer and let $\eps>0$.
  If $t\ge(1+\eps)n\log{n}/2$ and $b\ge (1+\eps)kn/2$
  then there exists a $(t,b)$-strategy $B$ of Builder such that
  \[
    \lim_{n\to\infty} \pr(B_t\text{ is $k$-connected}) = 1.
  \]
\end{theorem}
For $k=2$, we do not have a guess for the correct budget threshold at the hitting time
for $2$-connectivity, and believe it might be larger than $\co_2n=\frac{11}{8}n$.

For Hamilton cycles, we showed
(in \cref{thm:ham:hit}) that
at the hitting time and with an inflated budget,
Builder has a strategy that typically constructs a Hamilton cycle.
Together with \cref{thm:ham} it follows that
if time is (asymptotically) optimal and the budget is inflated (by a constant factor),
or when the budget is (asymptotically) optimal and the time is inflated,
then Builder has such a strategy.
As we mentioned in the introduction, Anastos~\cite{Ana22+} recently proved
that this is doable in asymptotically optimal time {\em and} budget (but {\em not} at the hitting time).
We remark that
for any $k=k(n)\ge 1$,
if $t\ll n^2/k^2$ and $b\le n+k$, then any $(t,b)$-strategy of Builder fails \whp{};
this
follows almost directly from \cite{FKSV16} (see Section~5 there).
This implies that if $t$ is asymptotically optimal
(or even $t=O(n\log{n})$)
then, in particular,
a budget of $n+o(\sqrt{n}/\log{n})$ does not suffice.

It might be interesting to extend either the result of Anastos or \cref{thm:ham:hit}
to other spanning structures.
For example, can Builder construct (\whp{}) the square of a Hamilton cycle at its hitting time
(or after its threshold)
using only a linear budget?
More generally, can Builder construct (\whp{}) any given bounded-degree spanning graph
slightly after its threshold
using only a linear budget?
It might also be interesting to optimise the constant $C$ in \cref{thm:ham:hit},
or find the minimum $k$ for which $\cO_{k,2}$ (or $\cO_k$) is \whp{} Hamiltonian.

\Cref{thm:trees,thm:cycles} considered the budget thresholds for trees and cycles.
It is not hard to extend our result on cycles to any fixed unicyclic graph:
after obtaining the cycle, the remaining forest can be constructed quickly and with a constant budget.
The smallest graph not covered by our results is therefore the diamond (a $K_4$ with one edge removed).
One of the difficulties in this particular problem
(or, possibly, in the case of any non-unicyclic graph)
is that several quite different natural strategies, some are naive and some more sophisticated,
turn out to give different upper bounds; each is superior in a different time regime.
We handled a similar difficulty when discussing cycles
but with lesser apparent complexity.
It would be interesting to
develop tools and approaches allowing to tackle the case of a generic fixed graph.

Finally, while {\em the} random graph process might be the most natural underlying graph process,
other processes may be considered.
One compelling example is the so-called {\em semi-random graph process}~\cite{BHKPSS20}
mentioned in the introduction.
The semi-random process has already exhibited some intriguing phenomena that significantly distinguish it from
the (standard) random graph process.
In our context, instead of observing a flow of random edges,
Builder sees a flow of random vertices
and, at each round, decides whether to connect the observed vertex to any other vertex.
Evidently, Builder can achieve connectivity with $t=b=n-1$;
answering questions concerning, e.g., Hamilton cycles or fixed subgraphs would be interesting.

\paragraph{Acknowledgement}
The authors thank the anonymous referees for their careful reading of the paper
and helpful suggestions that improved its presentation.

\bibliography{library}

\begin{bibdiv}
\begin{biblist}

\bib{AKS85}{incollection}{
      author={Ajtai, Mikl{\'o}s},
      author={Koml{\'o}s, J{\'a}nos},
      author={Szemer{\'e}di, Endre},
       title={First occurrence of {H}amilton cycles in random graphs},
        date={1985},
   booktitle={Cycles in graphs ({B}urnaby, {B}.{C}., 1982)},
      series={North-Holland Math. Stud.},
      volume={115},
   publisher={North-Holland, Amsterdam},
       pages={173\ndash 178},
      review={\MR{821516}},
}

\bib{Ana16+}{article}{
      author={{Anastos}, Michael},
       title={{Purchasing a $C_4$ online}},
        date={2016-11},
     journal={arXiv e-prints},
      eprint={1611.07503},
}

\bib{Ana22+}{article}{
      author={{Anastos}, Michael},
       title={{Constructing Hamilton cycles and perfect matchings
  efficiently}},
        date={2022-09},
     journal={arXiv e-prints},
      eprint={2209.09860},
}

\bib{ABKU99}{article}{
      author={Azar, Yossi},
      author={Broder, Andrei~Z.},
      author={Karlin, Anna~R.},
      author={Upfal, Eli},
       title={Balanced allocations},
        date={1999},
        ISSN={0097-5397},
     journal={SIAM Journal on Computing},
      volume={29},
      number={1},
       pages={180\ndash 200},
         url={https://doi-org.cmu.idm.oclc.org/10.1137/S0097539795288490},
      review={\MR{1710347}},
}

\bib{BMPR22+}{article}{
      author={{Behague}, Natalie~C.},
      author={{Marbach}, Trent~G.},
      author={{Pra{\l}at}, Pawel},
      author={{Ruci\'nski}, Andrzej},
       title={{Subgraph Games in the Semi-Random Graph Process and Its
  Generalization to Hypergraphs}},
        date={2021-05},
     journal={arXiv e-prints},
      eprint={2105.07034},
}

\bib{BKS12}{article}{
      author={Ben-Eliezer, Ido},
      author={Krivelevich, Michael},
      author={Sudakov, Benny},
       title={The size {R}amsey number of a directed path},
        date={2012},
        ISSN={0095-8956},
     journal={Journal of Combinatorial Theory. Series B},
      volume={102},
      number={3},
       pages={743\ndash 755},
         url={https://doi.org/10.1016/j.jctb.2011.10.002},
      review={\MR{2900815}},
}

\bib{BGHK20}{article}{
      author={Ben-Eliezer, Omri},
      author={Gishboliner, Lior},
      author={Hefetz, Dan},
      author={Krivelevich, Michael},
       title={Very fast construction of bounded-degree spanning graphs via the
  semi-random graph process},
        date={2020},
        ISSN={1042-9832},
     journal={Random Structures \& Algorithms},
      volume={57},
      number={4},
       pages={892\ndash 919},
         url={https://doi-org.cmu.idm.oclc.org/10.1002/rsa.20963},
      review={\MR{4170435}},
}

\bib{BHKPSS20}{article}{
      author={Ben-Eliezer, Omri},
      author={Hefetz, Dan},
      author={Kronenberg, Gal},
      author={Parczyk, Olaf},
      author={Shikhelman, Clara},
      author={Stojakovi\'{c}, Milo\v{s}},
       title={Semi-random graph process},
        date={2020},
        ISSN={1042-9832},
     journal={Random Structures \& Algorithms},
      volume={56},
      number={3},
       pages={648\ndash 675},
         url={https://doi-org.cmu.idm.oclc.org/10.1002/rsa.20887},
      review={\MR{4084186}},
}

\bib{BBFP09}{article}{
      author={Beveridge, Andrew},
      author={Bohman, Tom},
      author={Frieze, Alan},
      author={Pikhurko, Oleg},
       title={Memoryless rules for {A}chlioptas processes},
        date={2009},
        ISSN={0895-4801},
     journal={SIAM Journal on Discrete Mathematics},
      volume={23},
      number={2},
       pages={993\ndash 1008},
         url={https://doi-org.cmu.idm.oclc.org/10.1137/070684148},
      review={\MR{2519940}},
}

\bib{BF01}{article}{
      author={Bohman, Tom},
      author={Frieze, Alan},
       title={Avoiding a giant component},
        date={2001},
        ISSN={1042-9832},
     journal={Random Structures \& Algorithms},
      volume={19},
      number={1},
       pages={75\ndash 85},
         url={https://doi-org.cmu.idm.oclc.org/10.1002/rsa.1019},
      review={\MR{1848028}},
}

\bib{BF01add}{article}{
      author={Bohman, Tom},
      author={Frieze, Alan},
       title={Addendum to ``{A}voiding a giant component'' [{R}andom
  {S}tructures \& {A}lgorithms {\bf 19} (2001), no. 1, 75--85; mr1848028]},
        date={2002},
        ISSN={1042-9832},
     journal={Random Structures \& Algorithms},
      volume={20},
      number={1},
       pages={126\ndash 130},
         url={https://doi-org.cmu.idm.oclc.org/10.1002/rsa.10018},
      review={\MR{2160875}},
}

\bib{BFKLS11}{article}{
      author={Bohman, Tom},
      author={Frieze, Alan},
      author={Krivelevich, Michael},
      author={Loh, Po-Shen},
      author={Sudakov, Benny},
       title={Ramsey games with giants},
        date={2011},
        ISSN={1042-9832},
     journal={Random Structures \& Algorithms},
      volume={38},
      number={1-2},
       pages={1\ndash 32},
         url={https://doi-org.cmu.idm.oclc.org/10.1002/rsa.20343},
      review={\MR{2768882}},
}

\bib{BFW04}{article}{
      author={Bohman, Tom},
      author={Frieze, Alan},
      author={Wormald, Nicholas~C.},
       title={Avoidance of a giant component in half the edge set of a random
  graph},
        date={2004},
        ISSN={1042-9832},
     journal={Random Structures \& Algorithms},
      volume={25},
      number={4},
       pages={432\ndash 449},
         url={https://doi-org.cmu.idm.oclc.org/10.1002/rsa.20038},
      review={\MR{2099213}},
}

\bib{BK06off}{article}{
      author={Bohman, Tom},
      author={Kim, Jeong~Han},
       title={A phase transition for avoiding a giant component},
        date={2006},
        ISSN={1042-9832},
     journal={Random Structures \& Algorithms},
      volume={28},
      number={2},
       pages={195\ndash 214},
         url={https://doi-org.cmu.idm.oclc.org/10.1002/rsa.20085},
      review={\MR{2198497}},
}

\bib{BK06}{article}{
      author={Bohman, Tom},
      author={Kravitz, David},
       title={Creating a giant component},
        date={2006},
        ISSN={0963-5483},
     journal={Combinatorics, Probability and Computing},
      volume={15},
      number={4},
       pages={489\ndash 511},
         url={https://doi-org.cmu.idm.oclc.org/10.1017/S0963548306007486},
      review={\MR{2238042}},
}

\bib{Bol81}{incollection}{
      author={Bollob\'{a}s, B\'{e}la},
       title={Random graphs},
        date={1981},
   booktitle={Combinatorics ({S}wansea, 1981)},
      series={London Math. Soc. Lecture Note Ser.},
      volume={52},
   publisher={Cambridge Univ. Press, Cambridge-New York},
       pages={80\ndash 102},
      review={\MR{633650}},
}

\bib{Bol84}{incollection}{
      author={Bollob\'{a}s, B\'{e}la},
       title={The evolution of sparse graphs},
        date={1984},
   booktitle={Graph theory and combinatorics ({C}ambridge, 1983)},
   publisher={Academic Press, London},
       pages={35\ndash 57},
      review={\MR{777163}},
}

\bib{Bol}{book}{
      author={Bollob\'{a}s, B\'{e}la},
       title={Random graphs},
     edition={Second Edition},
      series={Cambridge Studies in Advanced Mathematics},
   publisher={Cambridge University Press, Cambridge},
        date={2001},
      volume={73},
        ISBN={0-521-80920-7; 0-521-79722-5},
         url={https://doi-org.cmu.idm.oclc.org/10.1017/CBO9780511814068},
      review={\MR{1864966}},
}

\bib{BFKLS18}{article}{
      author={Briggs, Joseph},
      author={Frieze, Alan},
      author={Krivelevich, Michael},
      author={Loh, Po-Shen},
      author={Sudakov, Benny},
       title={Packing {H}amilton cycles online},
        date={2018},
        ISSN={0963-5483},
     journal={Combinatorics, Probability and Computing},
      volume={27},
      number={4},
       pages={475\ndash 495},
         url={https://doi-org.cmu.idm.oclc.org/10.1017/S0963548318000159},
      review={\MR{3816058}},
}

\bib{CF95}{article}{
      author={Cooper, Colin},
      author={Frieze, Alan},
       title={On the connectivity of random {$k$}-th nearest neighbour graphs},
        date={1995},
        ISSN={0963-5483},
     journal={Combinatorics, Probability and Computing},
      volume={4},
      number={4},
       pages={343\ndash 362},
         url={https://doi-org.cmu.idm.oclc.org/10.1017/S0963548300001711},
      review={\MR{1377555}},
}

\bib{EPY97}{article}{
      author={Eppstein, David},
      author={Paterson, Michael~S.},
      author={Yao, Frances~Foong},
       title={On nearest-neighbor graphs},
        date={1997},
        ISSN={0179-5376},
     journal={Discrete \& Computational Geometry},
      volume={17},
      number={3},
       pages={263\ndash 282},
         url={https://doi-org.cmu.idm.oclc.org/10.1007/PL00009293},
      review={\MR{1432064}},
}

\bib{ER59}{article}{
      author={Erd\H{o}s, Paul},
      author={R\'{e}nyi, Alfr\'{e}d},
       title={On random graphs. {I}},
        date={1959},
        ISSN={0033-3883},
     journal={Publicationes Mathematicae Debrecen},
      volume={6},
       pages={290\ndash 297},
      review={\MR{120167}},
}

\bib{ER60}{article}{
      author={Erd\H{o}s, Paul},
      author={R\'{e}nyi, Alfr\'ed},
       title={On the evolution of random graphs},
        date={1960},
        ISSN={0541-9514},
     journal={A Magyar Tudom\'{a}nyos Akad\'{e}mia. Matematikai Kutat\'{o}
  Int\'{e}zet\'{e}nek K\"{o}zlem\'{e}nyei},
      volume={5},
       pages={17\ndash 61},
      review={\MR{125031}},
}

\bib{ER61}{article}{
      author={Erd\H{o}s, Paul},
      author={R\'{e}nyi, Alfr\'ed},
       title={On the strength of connectedness of a random graph},
        date={1961},
        ISSN={0001-5954},
     journal={Acta Mathematica. Academiae Scientiarum Hungaricae},
      volume={12},
       pages={261\ndash 267},
         url={https://doi-org.cmu.idm.oclc.org/10.1007/BF02066689},
      review={\MR{130187}},
}

\bib{FKSV16}{article}{
      author={Ferber, Asaf},
      author={Krivelevich, Michael},
      author={Sudakov, Benny},
      author={Vieira, Pedro},
       title={Finding {H}amilton cycles in random graphs with few queries},
        date={2016},
        ISSN={1042-9832},
     journal={Random Structures \& Algorithms},
      volume={49},
      number={4},
       pages={635\ndash 668},
         url={https://doi-org.cmu.idm.oclc.org/10.1002/rsa.20679},
      review={\MR{3570983}},
}

\bib{FFK06}{article}{
      author={Flaxman, Abraham~D.},
      author={Frieze, Alan},
      author={Krivelevich, Michael},
       title={On the random 2-stage minimum spanning tree},
        date={2006},
        ISSN={1042-9832},
     journal={Random Structures \& Algorithms},
      volume={28},
      number={1},
       pages={24\ndash 36},
         url={https://doi-org.cmu.idm.oclc.org/10.1002/rsa.20079},
      review={\MR{2187481}},
}

\bib{FGS05}{article}{
      author={Flaxman, Abraham~D.},
      author={Gamarnik, David},
      author={Sorkin, Gregory~B.},
       title={Embracing the giant component},
        date={2005},
        ISSN={1042-9832},
     journal={Random Structures \& Algorithms},
      volume={27},
      number={3},
       pages={277\ndash 289},
         url={https://doi-org.cmu.idm.oclc.org/10.1002/rsa.20070},
      review={\MR{2162599}},
}

\bib{FKRRT03}{article}{
      author={Friedgut, Ehud},
      author={Kohayakawa, Yoshiharu},
      author={R\"{o}dl, Vojt\v{e}ch},
      author={Ruci\'{n}ski, Andrzej},
      author={Tetali, Prasad},
       title={Ramsey games against a one-armed bandit},
        date={2003},
        ISSN={0963-5483},
     journal={Combinatorics, Probability and Computing},
      volume={12},
      number={5-6},
       pages={515\ndash 545},
         url={https://doi-org.cmu.idm.oclc.org/10.1017/S0963548303005881},
        note={Special issue on Ramsey theory},
      review={\MR{2037068}},
}

\bib{FriBib}{article}{
      author={{Frieze}, Alan},
       title={{Hamilton Cycles in Random Graphs: a bibliography}},
        date={2019-01},
     journal={arXiv e-prints},
       pages={[v23]},
      eprint={1901.07139v23},
}

\bib{FK}{book}{
      author={Frieze, Alan},
      author={Karo\'{n}ski, Micha\l},
       title={Introduction to random graphs},
   publisher={Cambridge University Press, Cambridge},
        date={2016},
        ISBN={978-1-107-11850-8},
         url={https://doi-org.cmu.idm.oclc.org/10.1017/CBO9781316339831},
      review={\MR{3675279}},
}

\bib{FP18}{article}{
      author={Frieze, Alan},
      author={Pegden, Wesley},
       title={Online purchasing under uncertainty},
        date={2018},
        ISSN={1042-9832},
     journal={Random Structures \& Algorithms},
      volume={53},
      number={2},
       pages={327\ndash 351},
         url={https://doi-org.cmu.idm.oclc.org/10.1002/rsa.20764},
      review={\MR{3845100}},
}

\bib{FP18cor}{article}{
      author={Frieze, Alan},
      author={Pegden, Wesley},
       title={Corrigendum to ``{O}nline purchasing under uncertainty''},
        date={2021},
        ISSN={1042-9832},
     journal={Random Structures \& Algorithms},
      volume={59},
      number={2},
       pages={288},
         url={https://doi-org.cmu.idm.oclc.org/10.1002/rsa.21012},
      review={\MR{4284174}},
}

\bib{GKMP22}{article}{
      author={Gao, Pu},
      author={Kami\'{n}ski, Bogumi\l},
      author={MacRury, Calum},
      author={Pra{\l}at, Pawe\l},
       title={Hamilton cycles in the semi-random graph process},
        date={2022},
        ISSN={0195-6698},
     journal={European Journal of Combinatorics},
      volume={99},
       pages={Paper No. 103423, 24},
         url={https://doi-org.cmu.idm.oclc.org/10.1016/j.ejc.2021.103423},
      review={\MR{4308013}},
}

\bib{GMP22+pm}{article}{
      author={Gao, Pu},
      author={MacRury, Calum},
      author={Pra\l~at, Pawe\l},
       title={Perfect matchings in the semirandom graph process},
        date={2022},
        ISSN={0895-4801},
     journal={SIAM Journal on Discrete Mathematics},
      volume={36},
      number={2},
       pages={1274\ndash 1290},
         url={https://doi-org.cmu.idm.oclc.org/10.1137/21M1446939},
      review={\MR{4431375}},
}

\bib{GMP22+ham}{article}{
      author={{Gao}, Pu},
      author={{MacRury}, Calum},
      author={{Pra{\l}at}, Pawel},
       title={{A Fully Adaptive Strategy for Hamiltonian Cycles in the
  Semi-Random Graph Process}},
        date={2022-05},
     journal={arXiv e-prints},
       pages={arXiv:2205.02350},
      eprint={2205.02350},
}

\bib{JLR}{book}{
      author={Janson, Svante},
      author={{\L}uczak, Tomasz},
      author={Ruci\'nski, Andrzej},
       title={Random graphs},
      series={Wiley-Interscience Series in Discrete Mathematics and
  Optimization},
   publisher={Wiley-Interscience, New York},
        date={2000},
        ISBN={0-471-17541-2},
         url={https://doi.org/10.1002/9781118032718},
      review={\MR{1782847}},
}

\bib{KS83}{article}{
      author={Koml\'{o}s, J\'{a}nos},
      author={Szemer\'{e}di, Endre},
       title={Limit distribution for the existence of {H}amiltonian cycles in a
  random graph},
        date={1983},
        ISSN={0012-365X},
     journal={Discrete Mathematics},
      volume={43},
      number={1},
       pages={55\ndash 63},
         url={https://doi-org.cmu.idm.oclc.org/10.1016/0012-365X(83)90021-3},
      review={\MR{680304}},
}

\bib{Kri16}{incollection}{
      author={Krivelevich, Michael},
       title={Long paths and hamiltonicity in random graphs},
        date={2016},
   booktitle={Random graphs, geometry and asymptotic structure},
      editor={Fountoulakis, Nikolaos},
      editor={Hefetz, Dan},
      series={London Mathematical Society Student Texts},
   publisher={Cambridge University Press},
       pages={4–27},
}

\bib{KLS09}{article}{
      author={Krivelevich, Michael},
      author={Loh, Po-Shen},
      author={Sudakov, Benny},
       title={Avoiding small subgraphs in {A}chlioptas processes},
        date={2009},
        ISSN={1042-9832},
     journal={Random Structures \& Algorithms},
      volume={34},
      number={1},
       pages={165\ndash 195},
         url={https://doi-org.cmu.idm.oclc.org/10.1002/rsa.20254},
      review={\MR{2478543}},
}

\bib{KLS10}{article}{
      author={Krivelevich, Michael},
      author={Lubetzky, Eyal},
      author={Sudakov, Benny},
       title={Hamiltonicity thresholds in {A}chlioptas processes},
        date={2010},
        ISSN={1042-9832},
     journal={Random Structures \& Algorithms},
      volume={37},
      number={1},
       pages={1\ndash 24},
         url={https://doi-org.cmu.idm.oclc.org/10.1002/rsa.20302},
      review={\MR{2674619}},
}

\bib{KS12}{article}{
      author={Krivelevich, Michael},
      author={Sp\"{o}hel, Reto},
       title={Creating small subgraphs in {A}chlioptas processes with growing
  parameter},
        date={2012},
        ISSN={0895-4801},
     journal={SIAM Journal on Discrete Mathematics},
      volume={26},
      number={2},
       pages={670\ndash 686},
         url={https://doi-org.cmu.idm.oclc.org/10.1137/110836146},
      review={\MR{2967491}},
}

\bib{Lic22+}{article}{
      author={{Lichev}, Lyuben},
       title={{$d$-connectivity of the random graph with restricted budget}},
        date={2022-08},
     journal={arXiv e-prints},
      eprint={2208.04111},
}

\bib{MSS09}{article}{
      author={Marciniszyn, Martin},
      author={Sp\"{o}hel, Reto},
      author={Steger, Angelika},
       title={Online {R}amsey games in random graphs},
        date={2009},
        ISSN={0963-5483},
     journal={Combinatorics, Probability and Computing},
      volume={18},
      number={1-2},
       pages={271\ndash 300},
         url={https://doi-org.cmu.idm.oclc.org/10.1017/S0963548308009632},
      review={\MR{2497384}},
}

\bib{MSS09ub}{article}{
      author={Marciniszyn, Martin},
      author={Sp\"{o}hel, Reto},
      author={Steger, Angelika},
       title={Upper bounds for online {R}amsey games in random graphs},
        date={2009},
        ISSN={0963-5483},
     journal={Combinatorics, Probability and Computing},
      volume={18},
      number={1-2},
       pages={259\ndash 270},
         url={https://doi-org.cmu.idm.oclc.org/10.1017/S0963548308009620},
      review={\MR{2497383}},
}

\bib{MST11}{article}{
      author={M\"{u}tze, Torsten},
      author={Sp\"{o}hel, Reto},
      author={Thomas, Henning},
       title={Small subgraphs in random graphs and the power of multiple
  choices},
        date={2011},
        ISSN={0095-8956},
     journal={Journal of Combinatorial Theory. Series B},
      volume={101},
      number={4},
       pages={237\ndash 268},
         url={https://doi-org.cmu.idm.oclc.org/10.1016/j.jctb.2010.12.008},
      review={\MR{2788094}},
}

\bib{Pos76}{article}{
      author={P\'{o}sa, Lajos},
       title={Hamiltonian circuits in random graphs},
        date={1976},
        ISSN={0012-365X},
     journal={Discrete Mathematics},
      volume={14},
      number={4},
       pages={359\ndash 364},
         url={https://doi.org/10.1016/0012-365X(76)90068-6},
      review={\MR{389666}},
}

\bib{SW07}{article}{
      author={Spencer, Joel},
      author={Wormald, Nicholas},
       title={Birth control for giants},
        date={2007},
        ISSN={0209-9683},
     journal={Combinatorica},
      volume={27},
      number={5},
       pages={587\ndash 628},
         url={https://doi-org.cmu.idm.oclc.org/10.1007/s00493-007-2163-2},
      review={\MR{2375718}},
}

\bib{SST10}{article}{
      author={Sp\"{o}hel, Reto},
      author={Steger, Angelika},
      author={Thomas, Henning},
       title={Coloring the edges of a random graph without a monochromatic
  giant component},
        date={2010},
     journal={Electronic Journal of Combinatorics},
      volume={17},
      number={1},
       pages={Research Paper 133, 7},
         url={http://www.combinatorics.org/Volume_17/Abstracts/v17i1r133.html},
      review={\MR{2729382}},
}

\end{biblist}
\end{bibdiv}

\end{document}